\documentclass[12pt]{amsart}
\usepackage{a4wide,enumerate,color,graphicx,comment}
\allowdisplaybreaks

\let\pa\partial  
\let\na\nabla  
\let\eps\varepsilon  
\newcommand{\N}{{\mathbb N}}  
\newcommand{\R}{{\mathbb R}} 
\newcommand{\diver}{\operatorname{div}}

\newtheorem{theorem}{Theorem}   
\newtheorem{lemma}[theorem]{Lemma}   
\newtheorem{proposition}[theorem]{Proposition}   
\newtheorem{remark}[theorem]{Remark}

\begin{document}  

\title[Semi-discrete Keller-Segel models]{Blow-up
of solutions to semi-discrete parabolic-elliptic Keller-Segel models}

\author{Ansgar J\"ungel}
\address{Institute for Analysis and Scientific Computing, Vienna University of  
	Technology, Wiedner Hauptstra\ss e 8--10, 1040 Wien, Austria}
\email{juengel@tuwien.ac.at} 

\author{Oliver Leingang}
\address{Institute for Analysis and Scientific Computing, Vienna University of  
	Technology, Wiedner Hauptstra\ss e 8--10, 1040 Wien, Austria}
\email{oliver.leingang@tuwien.ac.at} 

\date{\today}

\thanks{The authors acknowledge partial support from   
the Austrian Science Fund (FWF), grants P27352, P30000, F65, and W1245} 

\begin{abstract}
The existence of weak solutions and upper bounds for the blow-up time for
time-discrete parabolic-elliptic Keller-Segel models for chemotaxis in the 
two-dimensional whole space are proved. For various time discretizations,
including the implicit Euler, BDF, and Runge-Kutta methods, the same bounds
for the blow-up time as in the continuous case are derived by discrete 
versions of the virial argument. 
The theoretical results are illustrated by numerical simulations using an
upwind finite-element method combined with second-order time discretizations.
\end{abstract}

 

\maketitle


\section{Introduction}

This paper is concerned with the derivation of upper bounds for the blow-up
time in semi-discrete Keller-Segel systems in $\R^2$.
The (Patlak-) Keller-Segel model describes the evolution of the cell density of
bacteria or amoebae that are attracted by a chemical substance \cite{KeSe70,Pat53}.
Under simplifying assumptions, the cell density $n(x,t)$ and the density of
the chemo-attractant $c(x,t)$ satisfy the parabolic-elliptic equations
\begin{equation}\label{1.eq}
  \pa_t n = \diver(\na n-n\na c), \quad -\Delta c + \alpha c = n\quad\mbox{in }\R^2,
\end{equation}
where $\alpha\ge 0$ measures the degradation rate of the chemical substance. 
Denoting by $B_\alpha$ the Bessel potential if $\alpha>0$
and the Newton potential if $\alpha=0$ (see the appendix for the definitions), 
this system can be formulated more compactly as a single equation,
\begin{equation}\label{1.eqc}
  \pa_t n = \diver\big(\na n- n(\na B_\alpha*n)\big) \quad\mbox{in }\R^2
\end{equation}
with the initial condition
$$
  n(\cdot,0) = n_0\quad\mbox{in }\R^2.
$$

\subsection{Blow-up properties}

The coupling in equations \eqref{1.eq} is a positive feedback: The cells produce 
a signal that attracts other cells. The cell aggregation
is counterbalanced by diffusion, but if the cell density is
sufficiently large, the chemical interaction dominates diffusion and may lead
to finite-time blow-up of the cell density. 

System \eqref{1.eq} exhibits a dichotomy. Consider first the case without
degradation, $\alpha=0$. If
the initial mass satisfies $M:=\int_{\R^2}n_0dx < 8\pi$, no aggregation takes
place and the cells just diffuse for all times \cite{BDP06}. On the other hand, 
if the mass is supercritical, $M>8\pi$, and the second moment 
$I_0:=\int_{\R^2}n_0|x|^2dx$ is finite, the solutions blow up in finite time.
In the limit case $M=8\pi$, there is blow-up in infinite time with constant 
second moment \cite{BCM08}. 
We remark that the critical space is $L^{d/2}(\R^d)$ in space dimensions $d\ge 2$
\cite{CPZ04}. When the signal degradates, $\alpha>0$, a similar criterium holds:
The solutions exist for all time for subcritical initial masses
$M<8\pi$, and they blow up in finite time for supercritical masses $M>8\pi$ {\em and}
sufficiently small second moment $I_0$ \cite{CaCo08}. The blow-up time $T_{\rm max}$ 
can be estimated from above by 
\begin{equation}\label{Tstar}
  T_{\rm max}\le T^*:=\frac{2\pi I_0}{M(M-8\pi-2\sqrt{\alpha MI_0})}.
\end{equation}
The upper bound for the system with degradation becomes larger than that one
without degradation, indicating that blow-up
happens later than without degradation, which is biologically reasonable.
The idea of the proof is a virial argument: The second moment 
$I(t)=\int_{\R^2}n(\cdot,t)dx$ solves the differential inequality
\begin{equation}\label{1.dIdt}
  \frac{dI}{dt} \le -\frac{M}{2\pi}(M-8\pi) 
	+ \frac{\sqrt{\alpha}}{\pi}M^{3/2}\sqrt{I}.
\end{equation}
If $\alpha=0$, this expression becomes an identity. Now, if $t>T^*$, we
infer that $I(t)<0$, contradicting $n(\cdot,t)\ge 0$. 

Very detailed numerical simulations of the collapse phenomenon 
in the parabolic-para\-bol\-ic system
have been performed in, e.g., \cite{BCR05}. The asymptotic profile of blow-up
solutions in the parabolic-elliptic model was studied in \cite{DELS14}.
Numerical blow-up times were computed, for instance, 
in \cite{FMV15} using a kind of $H^2$ norm indicator; 
in \cite{BCR05} using the moving-mesh method; and
in \cite{Eps09} using discontinuous Galerkin approximations.

The aim of this paper is to prove the existence of weak solutions to various
versions of time-discrete equations 
and to derive discrete versions of inequality \eqref{1.dIdt}.
We analyze the implicit Euler and higher-order schemes, including BDF
(Backward Differentiation Formula) and Runge-Kutta schemes.
Our goal is not to design efficient numerical schemes but to 
{\em continue our program to ``translate'' mathematical techniques
from continuous to discrete situations} \cite{JuMi15,JuSc17,JuSc17a}. 

The proof of inequality \eqref{1.dIdt} is based on three properties: nonnegativity of
$n(x,t)$, mass conservation, and a symmetry argument in the nonlocal term in
\eqref{1.eqc}.
The implicit Euler scheme satisfies all these properties, while the nonnegativity
cannot be proved for higher-order schemes, although it seems to hold numerically
(see Section \ref{sec.num}). We discuss this issue below.

\subsection{State of the art}

The literature on the analysis and numerical approximation of Keller-Segel models 
is enormous. Therefore, 
we review only papers concerned with
the analysis of numerical schemes and the blow-up behavior of the discrete solutions, 
in particular possessing structure-preserving
features. We do not claim completeness and refer to the introduction of
\cite{AkAm17} for more references.

Most of the numerical schemes for the Keller-Segel model utilize the implicit Euler
method for the time discretization and aim to preserve some properties of the
continuous equations, like (local) mass conservation, positivity 
(more precisely, nonnegativity) preservation, or energy dissipation.
These schemes use
(semi-implicit) finite-difference methods \cite{CLM14,LWZ17,SaSu05}; 
an upwind finite-element discretization \cite{Sai07};
an Eulerian-Lagrangian scheme based on the characteristics method \cite{Smi07}; 
a Galerkin method with a diminishing flux limiter \cite{SSKHT13};
and finite-volume methods \cite{Fil06,ZhSa17}. A finite-volume scheme
was also studied in \cite{AkAm17} but with a first-order semi-exponentially 
fitted time discretization. Finally, a mass-transport steepest descent scheme
was analyzed in \cite{BCC08}. All these schemes are based on first-order
discretizations.

Only few results are concerned with higher-order time integrators.
Strongly $A(\theta)$-stable Runge-Kutta finite-element discretizations 
were analyzed in \cite{NaYa02}, and the convergence of
the discrete solution was shown. However, mass conservation or positivity
preservation was not verified.
As $A$-stable time integrators are computationally
very costly, often splitting methods are used.
A third-order strong stability preserving (SSP) Runge-Kutta time discretization 
for the advection term and the second-order 
Krylov IIF (implicit integration factor) method for the 
reaction-diffusion term, together with a positivity-preserving
discontinuous Galerkin approximation in space, 
was suggested in \cite{ZZLY16}, but without any analysis. 

A number of papers are concerned with the preservation of the nonnegativity 
of the cell density.
An example is \cite{ChKu08}, where a semi-discrete central-upwind scheme was proposed.
Moreover, in \cite{CEHK16}, a hybrid finite-volume
finite-difference method was combined with SSP explicit
Runge-Kutta schemes (for the parabolic-parabolic case) or the explicit Euler scheme
(for the parabolic-elliptic case). Clearly, a CFL condition is needed to ensure
the stability of the explicit schemes. 
Another approach was used in \cite{VaCh03} for a 
related tumor-angiogenesis model, where a Taylor series method in time 
allows for higher-order but still explicit schemes. 

We remark that there do not exist SSP implicit Runge-Kutta or multistep methods 
of higher order \cite[Section 6]{GST01}. Moreover, SSP for such methods is
guaranteed only under some finite time step condition \cite{BoCr79,Spi83}
(also see the recent work \cite{Ket11}).
In view of these results, positivity preservation of our higher-order schemes
cannot be expected.

We do not aim to preserve the free energy of the Keller-Segel system since
such schemes usually destroy the symmetry property needed for the blow-up argument;
see Remark \ref{rem.energy} for details. 
Some estimates on the discrete free energy in an implicit Euler
finite-volume approximation were shown in \cite{ZhSa17}.
A numerical scheme that dissipates the free energy numerically was suggested 
in \cite{CRW16} using a gradient-flow formulation of the energy functional 
with respect to a quadratic transportation distance. 
It is shown in \cite{SaSu05} that the dissipation of the discrete free energy 
may fail in an upwind finite-difference scheme.
The dissipation of the discrete entropy in a finite-volume modified
Keller-Segel system was proved in \cite{BeJu14}.

The novelty of this paper is the derivation of the critical initial mass and 
an upper bound for the blow-up time for various time discretizations, where
the values are the same as in the continuous case. This shows that the
``continuous'' methods carry over to the semi-discrete case. Because of the virial
argument, the extension to spatial discretizations is more delicate.
A possible direction is to extend the virial argument to a discrete Bessel 
potential, but we leave details to a future work.

\subsection{Main results}

Our main results can be sketched as follows (details will be given in 
sections \ref{sec.ie} and \ref{sec.hos}). Let $n_k$ be an approximation
of $n(\cdot,k\tau)$, where $\tau>0$ is the time step and $k\in\N$.
We recall the definitions $M=\int_{\R^2}n_0dx$ of the initial mass and
$I_0=\int_{\R^2}n_0|x|^2dx$ of the initial second moment.

\begin{itemize}
\item Existence of solutions to the implicit Euler scheme (Theorem \ref{thm.ie}):
For given $n_{k-1}\ge 0$ 
and sufficiently small time step $\tau$, there exists a unique weak solutions 
$n_k$ to the semi-discrete equations.
Moreover, the scheme preserves the positivity, conserves the mass, and has
finite second moment.
\item Finite-time blow-up for the implicit Euler scheme (Theorem \ref{thm.bu.ie}):
Let $M>8\pi$ and let $I_0$ be finite (if $\alpha=0$) or 
$I_0$ and $\tau$ be sufficiently small (if $\alpha>0$).
Then the semi-discrete solution exists only up to discrete times $k\tau\le T^*$, where
$T^*$ is defined in \eqref{Tstar}.
\item BDF schemes: For sufficiently small $\tau>0$, there exists a unique 
weak solution to the BDF-2 and BDF-3 scheme, 
conserving the mass and having finite second moment.
Moreover, under the same assumptions as for the implicit Euler scheme,
the solution blows up and $k\tau\le T^*$ (Theorem \ref{thm.bu.bdf2} for BDF-2 and
$\alpha\ge 0$, Theorem \ref{thm.bu.bdf3} for BDF-3 and $\alpha=0$).
\item Runge-Kutta schemes (Theorem \ref{thm.bu.rk}): If $\alpha=0$
and under the same assumptions as for the implicit Euler scheme,
the solution blows up and $k\tau\le T^*$. The same result holds for the
implicit midpoint and trapezoidal rule if $\alpha>0$.
\end{itemize}

The interesting fact is that the upper bound $T^*$ is the same for both the
continuous and semi-discrete equations. Observe that the existence results
do not require the condition $M<8\pi$, since they are local. 
In particular, the smallness condition on $\tau$ is natural, and the time step
needs to be chosen smaller and smaller when the blow-up time is approached.

The paper is organized as follows. Section \ref{sec.ie} is concerned with the
analysis of the implicit Euler scheme, while some higher-order schemes
(BDF, Runge-Kutta) are investigated in section \ref{sec.hos}.
In section \ref{sec.num} we provide some
numerical examples to illustrate our theoretical statements.
Finally, the appendix collects some auxiliary results.


\section{Implicit Euler scheme}\label{sec.ie}

We show the existence of weak solutions to the implicit Euler approximation 
of the Keller-Segel system \eqref{1.eqc} and their finite-time blow-up, where
\begin{equation}\label{2.ie}
  \frac{1}{\tau}(n_k-n_{k-1}) = \diver\big(\na n_k - n_k\na B_\alpha*n_k\big)
	\quad\mbox{in }\R^2.
\end{equation}
Here, $\tau>0$ is the time step and $B_\alpha$ is the Bessel potential
if $\alpha>0$ and the Newton potential if $\alpha=0$ (see the appendix).
First, we study the existence of solutions. For this, 
let $X:=L^1(\R^2)\cap L^\infty(\R^2)$ with norm
$\|u\|_X=\max\{\|u\|_{L^1(\R^2)},\|u\|_{L^\infty(\R^2)}\}$ for $u\in X$.

\begin{theorem}[Existence for the implicit Euler scheme]\label{thm.ie}
Let $n_{k-1}\in X$ and 
$$
  \tau < \frac{1}{(\pi+1/2)^2}\frac{1}{(\|n_{k-1}\|_X+1)^4}.
$$
Then there exists a unique weak solution $n_k\in X\cap H^1(\R^2)$ to \eqref{2.ie}
such that 
\begin{itemize}
\item nonnegativity: if $n_{k-1}\ge 0$ then $n_k\ge 0$ in $\R^2$,
\item conservation of mass: $\int_{\R^2}n_k dx = \int_{\R^2}n_{k-1}dx$,
\item control of second moment: if $\int_{\R^2} n_{k-1}|x|^2dx<\infty$ 
then $\int_{\R^2}n_k|x|^2dx<\infty$.
\end{itemize}
\end{theorem}

\begin{proof}
The strategy is to prove first the existence of a unique very weak
solution $n_k\in X$ to the truncated problem
\begin{equation}\label{2.aux}
  \frac{1}{\tau}\int_{\R^2}(n_k-n_{k-1})\phi dx
	= \int_{\R^2}n_k\Delta\phi dx + \int_{\R^2}n_k^+(\na B_\alpha*n_k)\cdot\na\phi dx
\end{equation}
for all $\phi\in H^1(\R^2)$, where $n_k^+=\max\{0,n_k\}$. The second step
is to show that $n_k\in H^1(\R^2)$. Then we can use $n_k^-=\min\{0,n_k\}$ as
a test function in the weak formulation and prove that $n_k\ge 0$.

{\em Step 1: Solution to \eqref{2.aux}.}
To simplify the notation, we write $n:=n_k$ and $n_0:=n_{k-1}$. We introduce
the operator $\na c:X\to L^\infty(\R^2)^2$, $\na c[n]=\na B_\alpha*n$. 
We claim that this operator is well-defined and continuous. 
Indeed, set for $\alpha\ge 0$,
\begin{equation}\label{2.ga}
  g_\alpha(r) := \int_0^\infty e^{-s-\alpha r^2/(4s)}dx\le 1
  \quad\mbox{for }r>0. 
\end{equation}	
Then, after the substitution $s=|x|^2/(4t)$, we can write
\begin{equation}\label{2.naB}
  \na B_\alpha(z) = -\frac{1}{2\pi}\frac{z}{|z|^2}g_\alpha(|z|), \quad z\in\R^2,
\end{equation}
and it follows that
\begin{align}
  |\na c[n](x)|
	&= \bigg|\frac{1}{2\pi}\int_{\R^2}\frac{x-y}{|x-y|^2}g_\alpha(|x-y|)n(y)dy\bigg| 
	\nonumber \\
	&\le \frac{1}{2\pi}\int_{|x-y|\le 1}\frac{|n(y)|}{|x-y|}dy
	+ \frac{1}{2\pi}\int_{|x-y|>1}\frac{|n(y)|}{|x-y|}dy \nonumber \\
  &\le \|n\|_{L^\infty(\R^2)} + \frac{1}{2\pi}\|n\|_{L^1(\R^2)}
	\le b\|n\|_X. \label{2.nac}
\end{align}
where $b:=1+1/(2\pi)$.
This shows the continuity of $\na c$.

Next, for given $\widetilde n\in X$, the linear problem
$$
  -\Delta n + \tau^{-1}(n-n_0) = -\diver(\widetilde n^+\na c[\widetilde n])
$$
has a unique solution in $H^1(\R^2)$. 
Indeed, since we have $\widetilde n^+\in L^2(\R^2)$,
$\na c[\widetilde n]\in L^\infty(\R^2)^2$, and therefore
$f:=\widetilde n^+\na c[\widetilde n]\in L^2(\R^2)^2$, we can apply Lemma 
\ref{lem.ell} in the appendix 
yielding the unique solvability of the linear problem. The solution is given by
\begin{equation}\label{2.n}
  n = \frac{1}{\tau}B_{1/\tau}*n_0 - \na B_{1/\tau}*(\widetilde n^+
	\na c[\widetilde n]).
\end{equation}
Hence, we can define the fixed-point operator $T:X\to X$ by
$T[\widetilde n]:=n$, and $n$ is given by \eqref{2.n}. Clearly, any fixed point of
$T$ is a  solution to \eqref{2.aux}.
We apply the Banach fixed-point theorem to $T$ on the set 
$S:=\{n\in X:\|n\|_X\le \|n_0\|_X+1\}$. For this, we need to show that 
$T:S\to S$ is a contraction.

For the proof of $T(S)\subset S$, we use Lemma \ref{lem.bessel}, the Young
inequality with $p=q=r=1$ (see the appendix), and \eqref{2.nac}:
\begin{align*}
  \|T[n]\|_{L^1(\R^2)}
	&\le \frac{1}{\tau}\|B_{1/\tau}* n_0\|_{L^1(\R^2)}
	+ \|\na B_{1/\tau}*(n^+\na c[n])\|_{L^1(\R^2)} \\
	&\le \frac{1}{\tau}\|B_{1/\tau}\|_{L^1(\R^2)}\|n_0\|_{L^1(\R^2)} 
	+ \|\na B_{1/\tau}\|_{L^1(\R^2)}\|\na c[n]\|_{L^\infty(\R^2)}\|n\|_{L^1(\R^2)} \\
	&\le \|n_0\|_{L^1(\R^2)} 
	+ \frac{\pi}{2}\tau^{1/2}b\|n\|_X^2.
\end{align*}
Similarly, we obtain
$$
  \|T[n]\|_{L^\infty(\R^2)}
	\le \|n_0\|_{L^\infty(\R^2)} + \frac{\pi}{2}\tau^{1/2}b\|n\|_X^2.
$$
Combining the previous two estimates, we conclude that
$$
  \|T[n]\|_X 	\le \|n_0\|_X + \frac{\pi}{2}\tau^{1/2}b\|n\|_X^2.
$$
Then choosing $\tau<(\pi b)^{-2}(\|n_0\|_X+1)^{-4}$, we see that
$\|T[n]\|_X\le \|n_0\|_X+1/2$, which shows that $T[n]\in S$.

To show the contraction property, let $n$, $m\in S$. Then, estimating as above,
\begin{align*}
  \|T[n]-T[m]\|_X
	&\le \big\|\na B_{1/\tau}*\big(n^+\na c[n]-m^+\na c[m]\big)\big\|_X \\
	&\le \|\na B_{1/\tau}\|_{L^1(\R^2)}\big\|(n^+-m^+)\na c[n] + m^+(\na c[n]-\na c[m])
	\big\|_X \\
  &\le \frac{\pi}{2}\tau^{1/2}\big(\|n^+-m^+\|_X\|\na c[n]\|_{L^\infty(\R^2)}
	+ \|m\|_X\|\na c[n-m]\|_{L^\infty(\R^2)}\big) \\
	&\le \pi\tau^{1/2}b\big(\|n_0\|_X+1\big)\|n-m\|_X.
\end{align*}
Since $\pi\tau^{1/2}b(\|n_0\|_X+1)<1$, $T$ is a contraction.
By the Banach fixed-point theorem, $T$ has a fixed point $n_k\in X$, which
is a solution to \eqref{2.aux}.

{\em Step 2: Regularity of the solution to \eqref{2.aux}.} We prove that for any
$\alpha\ge 0$, $n_k\in H^1(\R^2)$ solves
\begin{equation}\label{2.aux2}
  \frac{1}{\tau}\int_{\R^2}(n_k-n_{k-1})\phi dx
	= -\int_{\R^2}\na n_k\cdot\na\phi dx + \int_{\R^2}n_k^+(\na B_\alpha*n_k)\cdot
	\na\phi dx
\end{equation}
for all $\phi\in H^1(\R^2)$ and satisfies
\begin{equation}\label{2.est}
  \|n_k\|_{H^1(\R^2)} \le C_0\|n_{k-1}\|_{L^2(\R^2)}
\end{equation}
for some positive constant $C_0$ which depends on $\tau$ and $\|n_{k-1}\|_X$.

Again, we set $n:=n_k$ and $n_0:=n_{k-1}$. Since $n\in X\subset L^2(\R^2)$,
we have $n^+\na c[n]\in L^2(\R^2)^2$. Therefore, by Lemma \ref{lem.ell},
$n\in H^1(\R^2)$ is the unique solution to \eqref{2.aux2}. We take $n$ as a test 
function in that equation and use the Young inequality:
\begin{align*}
  \frac{1}{2\tau}\int_{\R^2}(n^2-n_0^2)dx
	&\le \frac{1}{\tau}\int_{\R^2}(n-n_0)n dx
	= -\|\na n\|_{L^2(\R^2)}^2 + \int_{\R^2}n^+\na c[n]\cdot\na n dx \\
	&\le -\frac12\|\na n\|_{L^2(\R^2)}^2 + \frac12\|n\|_{L^2(\R^2)}^2
	\|\na c[n]\|_{L^\infty(\R^2)}^2.
\end{align*}
By \eqref{2.nac} and $\|n\|_X\le\|n_0\|_X+1$, we find that
$$
  \frac{1}{\tau}\int_{\R^2}(n^2-n_0^2)dx + \|\na n\|_{L^2(\R^2)}^2
	\le b^2\|n\|_X^2\|n\|_{L^2(\R^2)}^2 
	\le b^2\big(\|n_0\|_X+1\big)^2\|n\|_{L^2(\R^2)}^2.
$$
Then, since $\beta:=1-b^2(\|n_0\|_X+1)^2\tau>0$, we infer that
$$
  \|n\|_{L^2(\R^2)}^2 \le \frac{1}{\beta}\|n_0\|_{L^2(\R^2)}^2,
$$
and the claim follows with $C_0=1/\sqrt{\beta}$.

{\em Step 3: Nonnegativity of $n_k$.}
Let $n_{k-1}\ge 0$ in $\R^2$.
We use $n_k^-=\min\{0,n_k\}\in H^1(\R^2)$ as a test function in \eqref{2.aux2}:
\begin{align*}
  \frac{1}{\tau}\int_{\R^2}(n_k^-)^2 dx
  &\le \frac{1}{\tau}\int_{\R^2}(n_k - n_{k-1})n_k^-dx \\
	&= -\int_{\R^2}|\na n_k^-|^2 dx + \int_{\R^2}n_k^+\na c[n_k]\cdot\na n_k^- dx
	\le 0,
\end{align*}
since $n_{k-1}n_k^-\le 0$ and the last integral on the right-hand side vanishes.
This shows that $n_k\ge 0$ in $\R^2$, and $n_k^+=n_k$ in \eqref{2.aux2}.

{\em Step 4: Mass conservation.}
The statement follows immediately if we could use $\phi(x)=1$ as a test function 
in \eqref{2.aux}. Since this function is not integrable, we need to approximate.
As in \cite{KoSu08}, we introduce the radially symmetric cut-off function
$\phi_R(x)=\phi(|x|/R)$, where $R\ge 1$ and
\begin{equation}\label{2.cutoff}
  \phi(r) = \begin{cases} 
	    1 & \mbox{for }0\leq r \leq 1, \\
      1-2(r-1)^2 & \mbox{for }1< r\leq 3/2, \\
      2(2-r)^2 & \mbox{for }3/2 < r \leq 2, \\
      0 & \mbox{for }r\geq 2.
   \end{cases}
\end{equation}
The following properties hold:
\begin{equation}\label{2.phi}
\begin{aligned}
  & \phi_R\in H^2(\R^2), \quad 
	\lim_{R\to\infty} \phi_R(x)=1 &&\quad\mbox{for all }x\in\R^2, \\
	& |\na\phi_R(x)-\na\phi_R(y)| \le \frac{C_1}{R^2}|x-y|, \quad
	|\Delta\phi_R(x)|\le \frac{C_2}{R^2} &&\quad\mbox{for all }x,y\in\R^2,
\end{aligned}
\end{equation}
for some constants $C_1$, $C_2>0$. Let $0\le\rho^\eps\in C_0^\infty(\R^2)$ be a 
standard mollifier and set $\phi^\eps_R=\phi_R*\rho^\eps$. Then
$\phi^\eps_R\to\phi_R$ in $H^2(\R^2)$ as $\eps\to 0$ 
\cite[Lemma 1.8.2]{Kry08}. Consequently, up to a subsequence,
which is not relabeled, $\phi^\eps_R\to\phi_R$, $\na\phi^\eps_R\to\na\phi_R$,
and $\Delta\phi^\eps_R\to\Delta\phi_R$ a.e.\ in $\R^2$.
We use $\phi^\eps_R$ as a test function in \eqref{2.aux} and insert \eqref{2.naB}:
\begin{align*}
  \bigg|\frac{1}{\tau}&\int_{\R^2}(n_k-n_{k-1})\phi^\eps_R dx\bigg|
	= \bigg|\int_{\R^2}n_k\Delta\phi^\eps_R dx
	+ \int_{\R^2}n_k(\na B_\alpha*n_k)\cdot\na\phi^\eps_R dx\bigg| \\
	&\le \frac{C_2}{R^2}\|n_k\|_{L^1(\R^2)}
	+ \frac{1}{2\pi}\bigg|\int_{\R^2}\int_{\R^2}n_k(x)n_k(y)g_\alpha(|x-y|)
	\frac{x-y}{|x-y|^2}\cdot\na\phi^\eps_R(x)dxdy\bigg|.
\end{align*}
By symmetry and \eqref{2.phi}, the second integral can be estimated as
\begin{align*}
  \frac{1}{2\pi}&\bigg|\int_{\R^2}	\int_{\R^2}n_k(x)n_k(y)g_\alpha(|x-y|)
	\frac{x-y}{|x-y|^2}\cdot\na\phi^\eps_R(x)dxdy\bigg| \\
	&= \frac{1}{4\pi}
	\bigg|\int_{\R^2}\int_{\R^2}n_k(x)n_k(y)g_\alpha(|x-y|)\frac{x-y}{|x-y|^2}
	\cdot\big(\na\phi^\eps_R(x)-\na\phi^\eps_R(y)\big)dydx\bigg| \\
	&\le \frac{C_1}{4\pi R^2}\int_{\R^2}\int_{\R^2}n_k(x)n_k(y)dydx 
	= \frac{C_1}{4\pi R^2}\|n_k\|_{L^1(\R^2)}^2.
\end{align*}
These estimates allow us to apply the dominated convergence theorem, which leads,
in the limit $\eps\to 0$, to
$$
  \frac{1}{\tau}\int_{\R^2}(n_k-n_{k-1})\phi_R dx
	= \int_{\R^2}n_k\Delta\phi_R dx
	+ \int_{\R^2}n_k(\na B_\alpha*n_k)\cdot\na\phi_R dx.
$$
The same estimates as before show that both integrals on the right-hand side
can be estimated by a multiple of $1/R^2$ such that the limit $R\to\infty$ leads to
$$
  \frac{1}{\tau}\int_{\R^2}(n_k-n_{k-1})dx = 0,
$$
which gives mass conservation.

{\em Step 5: Control of the second moment.}
Similarly as in step 4, we approximate $|x|^2$ by setting 
$\psi_R(x)=|x|^2\phi_R(x)$, where $\phi_R$ is defined in \eqref{2.cutoff}.
Then $\psi_R\in H^2(\R^2)$, $\na\psi_R$ is Lipschitz continuous on $\R^2$,
and $\Delta\psi_R$ is bounded. Taking a standard mollifier $\rho^\eps\ge 0$, we set
$\psi^\eps_R=\psi_R*\rho^\eps\in C_0^\infty(\R^2)$. Using $\psi_R^\eps$ as
a test function in \eqref{2.aux} and passing to the limit $\eps\to 0$ and 
then $R\to\infty$, it follows that
\begin{equation}\label{2.Ik}
  \frac{1}{\tau}\int_{\R^2}(n_k-n_{k-1})|x|^2 dx
	= 4\int_{\R^2}n_k dx + 2\int_{\R^2}n_k(\na B_\alpha*n_k)\cdot x dx.
\end{equation}
Young's inequality and estimate \eqref{2.nac} for $\na c[n_k]=\na B_\alpha*n_k$ 
show that
\begin{align*}
  2\,\bigg|\int_{\R^2}n_k(\na B_\alpha*n_k)\cdot x dx\bigg|
	&\le 2b\|n_k\|_X\int_{\R^2}n_k|x|dx \\
	&\le b\|n_k\|_X\int_{\R^2}n_k dx + b\|n_k\|_X\int_{\R^2}n_k|x|^2 dx. 
\end{align*}
Therefore, with $\|n_k\|_X\le\|n_{k-1}\|_X+1$, \eqref{2.Ik} gives
$$
  \big(1-\tau b(\|n_{k-1}\|_X+1)\big)\int_{\R^2}n_k|x|^2 dx
	\le \int_{\R^2}n_{k-1}|x|^2 dx + \tau(4+b\|n_k\|_X)\int_{\R^2}n_0 dx.
$$
Since $1-\tau b(\|n_{k-1}\|_X+1)>0$, we infer that the second moment of $n_k$
is bounded if the second moment of $n_{k-1}$ does so.
\end{proof}

Next, we turn to the finite-time blow up of semi-discrete solutions.
Set, for $\alpha>0$,
\begin{equation}\label{Istar}
   I^* := \frac{(M-8\pi)^2}{4\alpha M}, \quad
	 \tau^* := \frac{\pi(M-8\pi)}{2\alpha M^2}.
\end{equation}

\begin{theorem}[Blow-up for the implicit Euler scheme]\label{thm.bu.ie}
Assume that
$$
  n_0\ge 0, \quad I_0 := \int_{\R^2}n_0(x)|x|^2dx < \infty, \quad
	M := \int_{\R^2}n_0dx > 8\pi.
$$
Let $(n_k)\subset L^1(\R^2)\cap H^1(\R^2)$ 
be a sequence of nonnegative weak solutions to \eqref{2.ie}. 
Then this sequence is finite with maximal index $k_{\rm{max}}$, where, if $\alpha=0$,
\begin{equation}\label{3.kmax1}
  k_{\rm{max}} \le \frac{2\pi I_0}{\tau M(M-8\pi)}.
\end{equation}
In case $\alpha>0$, if additionally $I_0\le I^*$ and $\tau<\tau^*$ then
\begin{equation}\label{3.kmax2}
	k_{\rm{max}} \le \frac{2\pi I_0}{\tau M(M-8\pi - 2\sqrt{\alpha MI_0})}.
\end{equation}
\end{theorem}

\begin{proof}
Let first $\alpha=0$ and let $n_k$ be a weak solution to \eqref{2.ie} with
$k>2\pi I_0/(\tau M(M-8\pi))$, i.e., we assume that \eqref{3.kmax1} does not hold. 
We set $I_k=\int_{\R^2}n_k|x|^2 dx$. Then, by \eqref{2.Ik}, for any $j\le k$,
\begin{align*}
  I_j - I_{j-1} &= 4\tau\int_{\R^2}n_jdx + 2\tau\int_{\R^2}n_j(\na B_0*n_j)
	\cdot xdx \\
	&= 4\tau M - \frac{\tau}{\pi}\int_{\R^2}\int_{\R^2}
	n_j(x)\frac{x\cdot(x-y)}{|x-y|^2}n_j(y)dydx,
\end{align*}
where we used the conservation of mass and the definition of $\na B_0*n_k$.
A symmetry argument leads to
\begin{align*}
  I_j - I_{j-1} &= 4\tau M - \frac{\tau}{2\pi}\int_{\R^2}\int_{\R^2}
	\frac{(x-y)\cdot(x-y)}{|x-y|^2}n_j(x)n_j(y)dydx \\
	&= 4\tau M - \frac{\tau}{2\pi}M^2 = \frac{\tau M}{2\pi}(8\pi-M).
\end{align*}
Summing this identity over $j=1,\ldots,k$ and taking into account the choice of $k$
gives
$$
  I_k = I_0 - \frac{k\tau M}{2\pi}(M-8\pi) < 0,
$$
which is a contradiction to $n_k\ge 0$.

Next, let $\alpha>0$. For the proof, we follow the lines of \cite[Section~6]{CaCo08}
but the end of the proof is different. Let $n_k\ge 0$ be a weak solution 
to \eqref{2.ie} such that \eqref{3.kmax2} is not true. 
Similarly as above, we find that
\begin{align}
  I_k - I_{k-1} &= 4\tau\int_{\R^2}n_kdx + 2\tau\int_{\R^2}n_k(\na B_\alpha*n_k)
	\cdot x dx \nonumber \\
	&= 4\tau M - \frac{\tau}{2\pi}\int_{\R^2}\int_{\R^2}g_\alpha(|x-y|)n_k(x)n_k(y)dydx 
	\nonumber \\
	&= \frac{\tau M}{2\pi}(8\pi-M) + \frac{\tau}{2\pi}\int_{\R^2}\int_{\R^2}
	\big(1-g_\alpha(|x-y|)\big)n_k(x)n_k(y)dydx, \label{3.aux}
\end{align}
where we recall definition \eqref{2.ga} of $g_\alpha$.

We need to estimate $1-g_\alpha(r)$. For this, let $z\in\R^2$, 
$r=|z|\in(0,1/\sqrt{\alpha})$. We compute
$$
  \frac{d}{dr}(1-g_\alpha(r)) = \frac{\alpha r}{2}\int_0^\infty \frac{1}{s}
	e^{-s-\alpha r^2/(4s)}ds = 2\pi\alpha |z| B_1(\sqrt{\alpha}z) 
	\le \sqrt{\alpha}K,
$$
where $K=2\pi\sup_{|x|<1}|x| B_1(|x|)$. It is known that $B_1(x)$ behaves
asymptotically as $-\log |x|$ as $|x|\to 0$, so $K$ is finite. 
A numerical computation shows that $\sup_{|x|<1}|x| B_1(|x|)\approx 0.0742$ and
$K\approx 0.4662$. We conclude that
$$
  0\le 1-g_\alpha(|z|)\le \sqrt{\alpha}K|z| \quad\mbox{for }0<\sqrt{\alpha}|z|<1.
$$
This bound, together with $1-g_\alpha(|z|)\le 1$, shows that the
last integral in \eqref{3.aux} can be estimated as follows:
\begin{align}
  \frac{\tau}{2\pi} & \int\int_{\{\sqrt{\alpha}|x-y|<1\}}
	\big(1-g_\alpha(|x-y|)\big)n_k(x)n_k(y)dydx \nonumber \\
	&\phantom{xx}{}
	+ \frac{\tau}{2\pi} \int\int_{\{\sqrt{\alpha}|x-y|\ge 1\}}
	\big(1-g_\alpha(|x-y|)\big)n_k(x)n_k(y)dydx \nonumber \\
	&\le \tau\frac{\sqrt{\alpha}}{2\pi}\max\{1,K\}
	\int_{\R^2}\int_{\R^2}|x-y|n_k(x)n_k(y)dydx \nonumber \\
	&\le \tau\frac{\sqrt{\alpha}}{\pi}M\int_{\R^2}|y|n_k(y)dy
	\le \tau\frac{\sqrt{\alpha}}{\pi} M^{3/2}I_k^{1/2}, \label{2.estg}
\end{align}
where 
we have applied the Cauchy-Schwarz inequality 
in the last step. We infer from \eqref{3.aux} that
$$
  I_k-I_{k-1} \le \frac{\tau}{2\pi}M(8\pi-M) 
	+ \tau\frac{\sqrt{\alpha}}{\pi} M^{3/2}I_k^{1/2}.
$$

Now, the argument differs from that one used in \cite{CaCo08}. Set
$\beta=\sqrt{\alpha}M^{3/2}/\pi$ and $\gamma=M(M-8\pi)/(2\pi)$. Then
we need to solve the recursive inequality
\begin{equation}\label{3.aux2}
  I_k-I_{k-1} \le \tau f(I_k) := \tau\big(\beta I_k^{1/2}-\gamma\big).
\end{equation}
By definition of $I^*$, we have $f(I^*)=0$. Since $f$ is increasing and $I_0\le I^*$, 
it holds that $f(I_0)\le 0$. We proceed by induction. Let $f(I_{k-1})\le 0$.
We suppose that $f(I_k)>0$ and show that this leads to a contradiction.
Inequality \eqref{3.aux2} is equivalent to 
$$
  \frac{I_k^{1/2}-I_{k-1}^{1/2}}{I_k^{1/2}-\gamma/\beta}
	\frac{I_k^{1/2}+I_{k-1}^{1/2}}{\beta} \le \tau,
$$
Since $f(I_{k-1})\le 0$, the first factor is larger than or equal to one, 
and taking into account $f(I_k)>0$ or $I_k^{1/2}>\gamma/\beta$, we deduce that
$$
  \frac{\gamma}{\beta^2} < \frac{I_k^{1/2}}{\beta}
	\le \frac{I_k^{1/2}+I_{k-1}^{1/2}}{\beta} \le \tau,
$$
which contradicts the smallness condition $\tau<\tau^*=\gamma/\beta^2$.
We infer that $f(I_k)\le 0$. 

Then, summing \eqref{3.aux2} from $k=1,\ldots,j$,
$$
  I_j \le I_0 + \tau\sum_{k=1}^j f(I_k) \le I_0 + \tau jf(I_0).
$$
We deduce that $I_j$ becomes negative for $j>-I_0/(\tau f(I_0))$ which contradicts
$n_k\ge 0$. This completes the proof.
\end{proof}

\begin{remark}[Semi-discrete energy dissipation]\label{rem.energy}\rm
It is possible to design semi-discrete schemes that dissipate the
discrete free energy 
$$
  E_k = \int_{\R^2}\bigg(n_k(\log n_k-1) - \frac12 c_kn_k\bigg)dx
$$
(and conserve the mass and preserve the positivity).
An example, taken from \cite{SaSu05}, is the semi-implicit scheme
\begin{equation}\label{2.energy}
  \tau^{-1}(n_k-n_{k-1}) = \diver(n_k-n_k\na c_{k-1}), \quad
	-\Delta c_{k-1}+\alpha c_{k-1} = n_{k-1}\quad \mbox{in }\R^2.
\end{equation}
Indeed, by the convexity of $s\mapsto s\log s$, we obtain
\begin{align*}
  \int_{\R^2} & \big(n_k(\log n_k-1) - n_{k-1}(\log n_{k-1}-1)\big)dx \\
	&\le \int_{\R^2} (n_k-n_{k-1})\log n_k dx 
	= \tau\int_{\R^2}\bigg(-\frac{|\na n_k|^2}{n_k} + \na c_{k-1}\cdot\na n_k\bigg)dx.
\end{align*}
Furthermore, translating the computation in \cite[Section 5.2.1]{Per07} to the 
semi-discrete case,
\begin{align*}
  \frac12\int_{\R^2} & (n_kc_k-n_{k-1}c_{k-1})dx
	= \frac12\int_{\R^2}\big((n_k-n_{k-1})c_k + n_{k-1}(c_k-c_{k-1})\big)dx \\
	&= \frac12\int_{\R^2}\big((n_k-n_{k-1})c_k + \big(-\Delta(c_k-c_{k-1})
	+ \alpha(c_k-c_{k-1})\big)c_{k-1}\big)dx \\
	&= \int_{\R^2}(n_k-n_{k-1})c_{k-1} dx
	= -\tau\int_{\R^2}\big(\na n_k\cdot\na c_{k-1} + n_k|\na c_{k-1}|^2\big)dx.
\end{align*}
Subtracting the latter from the former expression, we conclude that
$$
  E_k-E_{k-1} \le -\tau\int_{\R^2} n_k|\na(\log n_k-c_{k-1})|^2 dx \le 0.
$$
Unfortunately, scheme \eqref{2.energy} does not allow us to apply the symmetrization
argument used in the proof of Theorem \ref{thm.bu.ie}, since the drift part
depends on two different time steps.
The question whether \eqref{2.energy} admits solutions that blow up in finite time
remains open.
\qed
\end{remark}


\section{Higher-order schemes}\label{sec.hos}

We investigate BDF and general Runge-Kutta schemes.

\subsection{BDF-2 scheme}

The scheme reads as
\begin{equation}\label{bdf2}
  \frac{1}{\tau}\bigg(\frac32 n_k - 2n_{k-1} + \frac12 n_{k-2}\bigg)
	= \diver\big(\na n_k - n_k(\na B_\alpha*n_k)\big) \quad\mbox{in }\R^2
\end{equation}
for $k\ge 2$, where $n_0$ is given and $n_1$ is computed
from $n_0$ using the implicit Euler scheme.

\begin{lemma}[Existence for the BDF-2 scheme]\label{thm.bdf2}
Let $\alpha\ge 0$, 
$n_{k-2}$, $n_{k-1}\in L^1(\R^2)\cap L^\infty(\R^2)$, and
$$
  \tau \le \frac32\frac{1}{(\pi+1/2)^2}\frac{1}{(1+\|2n_{k-1}-\frac12 n_{k-2}\|_X)^4}.
$$
Then there exists a weak solution $n_k\in L^1(\R^2)\cap L^\infty(\R^2)\cap H^1(\R^2)$
to \eqref{bdf2} with the following properties:
\begin{itemize}
\item conservation of mass: $\int_{\R^2}n_k dx = \int_{\R^2}n_{k-1}dx$,
\item control of second moment: if $\int_{\R^2}n_{k-1}|x|^2dx<\infty$ then
$\int_{\R^2}n_k|x|^2dx<\infty$.
\end{itemize}
\end{lemma}

\begin{proof}
The proof is exactly as for Theorem \ref{thm.ie} since we can formulate scheme
\eqref{bdf2} as
$$
  -\Delta n_k + \frac{3}{2\tau}n_k = \frac{1}{\tau}\bigg(2n_{k-1}-\frac12 n_{k-2}\bigg)
	- \diver(n_k\na B_\alpha*n_k),
$$
and the first term on the right-hand side plays the role of $n_{k-1}$ in the
implicit Euler scheme. The only difference to the proof of Theorem \ref{thm.ie} 
is that we replace $n_k^+$ in \eqref{2.aux} by $n_k$, since the truncation was
only needed to show the nonnegativity of $n_k$, which we are not able to show
for the BDF-2 scheme.
\end{proof}

\begin{theorem}[Blow-up for the BFD-2 scheme]\label{thm.bu.bdf2}
Assume that $\alpha\ge 0$ and
$$
  n_0\ge 0, \quad I_0 := \int_{\R^2}n_0(x)|x|^2dx < \infty, \quad
	M:=\int_{\R^2}n_0 dx > 8\pi.
$$
Let $(n_k)\subset L^1(\R^2)\cap H^1(\R^2)$ 
be a sequence of nonnegative weak solutions to \eqref{bdf2}. Then this
sequence is finite with maximal index $k_{\rm max}$, where $k_{\rm max}$ is
bounded from above according to \eqref{3.kmax1} (if $\alpha=0$) or
\eqref{3.kmax2} (if $\alpha>0$ and additionally $I_0\le I^*$ 
and $\tau\le\tau^*$, where $I^*$ and $\tau^*$ are defined in \eqref{Istar}).
\end{theorem}

\begin{proof}
The proof is similar to that one of Theorem \ref{thm.bu.ie} but the iteration
argument is different. First, let $\alpha=0$.
We know from the proof of Theorem \ref{thm.bu.ie} that
\begin{equation}\label{3.ie}
  I_1-I_0=-\tau\gamma, 
\end{equation}
where we recall that $\gamma=M(M-8\pi)/(2\pi)$. Using the same approximation
of $|x|^2$ as in step 5 of the proof of Theorem \ref{thm.bu.ie}, 
we can justify the weak formulation (see \eqref{2.Ik})
\begin{align*}
  \int_{\R^2}(n_j-n_{j-1})|x|^2dx
	&= \frac13\int_{\R^2}(n_{j-1}-n_{j-2})|x|^2 dx
	+ \frac83\tau\int_{\R^2}n_jdx \\
	&\phantom{xx}{}+ \frac23\tau\int_{\R^2}n_k(\na B_\alpha*n_k)\cdot xdx.
\end{align*}
The last integral can be calculated as in the proof of Theorem \ref{thm.bu.ie}
and we end up with
$$
  I_k-I_{k-1} = \frac13(I_{k-1}-I_{k-2}) + \frac23\frac{\tau}{2\pi}M(8\pi-M).
$$
We iterate this identity and insert \eqref{3.ie}:
$$
  I_k-I_{k-1} = \frac{1}{3^{k-1}}(I_1-I_0) 
	- \frac23\tau\gamma\sum_{j=1}^{k-1}\frac{1}{3^j}
	= -\frac{\tau\gamma}{3^{k-1}} - \tau\gamma\bigg(1-\frac{1}{3^{k-1}}\bigg)
	= -\tau\gamma.
$$
As in the proof of Theorem \ref{thm.bu.ie}, this leads to a contradiction for
large values of $k$. 

Next, let $\alpha>0$. As the first step is computed with the implicit Euler scheme,
the proof of Theorem \ref{thm.bu.ie} gives the estimate
$$
  I_1 - I_0 \le \tau f(I_1),
$$
where $f(s)=\beta\sqrt{s}-\gamma$ and $\beta=\sqrt{\alpha}M^{3/2}/\pi$. 
Moreover, since $f(I_0)\le 0$, we know that $I_1\le I_0$, and this gives $f(I_1)\le 0$.

For the following time steps, we obtain
\begin{equation}\label{3.aux3}
  I_k - I_{k-1} \le \frac13(I_{k-1}-I_{k-2}) + \frac{2\tau}{3}f(I_k), \quad k\ge 2.
\end{equation}
Let us assume, by induction, that $I_{k-1}\le I_{k-2}$ and $f(I_{k-1})\le 0$
for $k\ge 2$. We will prove that $I_k\le I_{k-1}$ and $f(I_k)\le 0$. 
Assume by contradiction that $f(I_k)>0$, which is equivalent to 
$I_k^{1/2}>\gamma/\beta$. Then, using
$I_{k-1}-I_{k-2}\le 0$, we reformulate \eqref{3.aux3} as
$$
  \frac{I_k^{1/2}-I_{k-1}^{1/2}}{I_k^{1/2}-\gamma/\beta}
	\frac{I_k^{1/2}+I_{k-1}^{1/2}}{\beta} \le \frac{2\tau}{3}.
$$
Since $f(I_{k-1})\le 0$, the first factor is larger than or equal to one, so
$$
  \frac{I_k^{1/2}+I_{k-1}^{1/2}}{\beta} \le \frac{2\tau}{3},
$$
and $\tau\le\gamma/\beta^2$ leads to
$$
  \frac{I_k^{1/2}}{\beta}
	\le \frac{I_k^{1/2}+I_{k-1}^{1/2}}{\beta} 
	\le \frac{2\tau}{3} \le \frac23\frac{\gamma}{\beta^2}
$$
or $I_k^{1/2}<\gamma/\beta$, which contradicts $f(I_k)>0$.
We conclude that $f(I_k)\le 0$ and therefore, by \eqref{3.aux3}, 
$I_k\le I_{k-1}\le 0$, showing the claim.

We infer that $f(I_k)\le f(I_{k-1})\le\cdots\le f(I_0)$, 
since $f$ is nondecreasing. Hence, again from \eqref{3.aux3} and using
$I_1-I_0\le \tau f(I_0)$,
\begin{align*}
  I_k-I_{k-1} &\le \frac13(I_{k-1}-I_{k-2}) + \frac{2\tau}{3}f(I_0) 
	\le \frac{1}{3^{k-1}}(I_1-I_0) + 2\tau \sum_{j=1}^{k-1}\frac{1}{3^j}f(I_0) \\
	&\le \frac{\tau}{3^{k-1}}f(I_0) + \tau\bigg(1-\frac{1}{3^{k-1}}\bigg)f(I_0)
	= \tau f(I_0).
\end{align*}
Thus, $I_k\le I_0+\tau\sum_{j=1}^k f(I_0)=I_0+\tau k f(I_0)$, and this
leads to the contradiction $I_k<0$ for sufficiently large $k\in\N$, 
completing the proof.
\end{proof}


\subsection{BDF-3 scheme}

The finite-time blow-up can be also shown for solutions to higher-order BDF
schemes, at least in the case $\alpha=0$. As an example, let us consider the
BDF-3 scheme
\begin{equation}\label{bdf3}
  \frac{1}{6\tau}(11n_k - 18n_{k-1} + 9n_{k-2} - 2n_{k-3})
	= \diver(\na n_k - n_k\na B_0*n_k) \quad\mbox{in }\R^2,
\end{equation}
where $n_{k-1}$, $n_{k-2}$, and $n_{k-3}$ are given.
The existence of solutions can be shown as in Lemma \ref{thm.bdf2}.
First, we prove that the scheme preserves the mass.

\begin{lemma}[Conservation of mass]\label{lem.BDF3}
Let $n_0$, $n_1$, and $n_2$ be given and having the same mass $M$. Then the
solution $n_k$ has the same mass, $\int_{\R^2}n_kdx=M$, for $k\ge 3$.
\end{lemma}

\begin{proof}
We proceed by induction. Employing the mollified version of the
cut-off function \eqref{2.cutoff} as a test function in \eqref{bdf3} 
and passing to the limits $\eps\to 0$ and $R\to\infty$, we arrive at
$$
  \frac{1}{6\tau}\int_{\R^2}(11n_k - 18n_{k-1} + 9n_{k-2} - 2n_{k-3})dx = 0.
$$
If $k=3$, this is equivalent to
$$
  \int_{\R^2}n_3 dx = \int_{\R^2}\bigg(\frac{18}{11}n_2 - \frac{9}{11}n_1
	+ \frac{2}{11}n_0\bigg)dx = M,
$$
since $n_2$, $n_1$, and $n_0$ have the same mass $M$. For the induction step,
if $n_{k-1}$, $n_{k-2}$, and $n_{k-3}$ for $k>4$ have the same mass, the same
argument as above shows that $\int_{\R^2}n_k=M$.
\end{proof}

We recall that $I_k=\int_{\R^2}n_k|x|^2dx$ for $k\in\N_0$ 
and $\gamma=M(M-8\pi)/(2\pi)$.

\begin{theorem}[Blow-up for the BDF-3 scheme]\label{thm.bu.bdf3}
Assume that $\alpha=0$, $I_2-I_1=I_1-I_0=-\tau\gamma$, and
$$
  n_0\ge 0, \quad I_0 := \int_{\R^2}n_0(x)|x|^2dx < \infty, \quad
	M:=\int_{\R^2}n_0 dx > 8\pi.
$$
Let $(n_k)\subset L^1(\R^2)\cap H^1(\R^2)$ 
be a sequence of nonnegative weak solutions to \eqref{bdf3}. Then this
sequence is finite with maximal index $k_{\rm max}$, where $k_{\rm max}$ is
bounded from above according to \eqref{3.kmax1} (if $\alpha=0$) or
\eqref{3.kmax2} (if $\alpha>0$ and additionally $I_0\le I^*$ 
and $\tau\le\tau^*$, where $I^*$ and $\tau^*$ are defined in \eqref{Istar}).
\end{theorem}

\begin{proof}
We claim that $I_k-I_{k-1}=-\tau\gamma$. To prove this, 
we proceed by induction. Let $k=3$. We take an approximation of $|x|^2$
as a test function in \eqref{bdf3}. Then, arguing as in the previous sections,
we find that
\begin{align*}
  \frac{11}{6} & (I_3-I_2) - \frac76(I_2-I_1) + \frac13(I_1-I_0) \\
  &= \frac{1}{6}(11I_3 - 18I_2 + 9I_1 - 2I_0) 
	= \frac{\tau}{2\pi}M(8\pi-M) = -\tau\gamma.
\end{align*}
Since  $I_2-I_1=I_1-I_0=-\tau\gamma$, it follows that
$$
  \frac{11}{6}(I_3-I_2) = -\frac76\tau\gamma + \frac13\tau\gamma - \tau\gamma
	= -\frac{11}{6}\tau\gamma.
$$
For the induction step, we assume that $I_{k-1}-I_{k-2}=I_{k-2}-I_{k-3}=-\tau\gamma$
for $k>3$. Then, as above,
$$
  \frac{11}{6}(I_k-I_{k-1}) - \frac76(I_{k-1}-I_{k-2}) + \frac13(I_{k-2}-I_{k-3})
	= -\tau\gamma.
$$
which shows that $I_k-I_{k-1}=-\tau\gamma$.
As in the proof of Theorem \ref{thm.bu.ie}, this leads to a contradiction for
large values of $k$. 
\end{proof}

The previous proof can be generalized to all BDF-$m$ methods
$$
  \frac{1}{\tau}\sum_{i=0}^m a_i n_{k-i} = \diver(\na n_k-\na B_\alpha*n_k),
$$
where $a_i\in\R$ satisfy $\sum_{i=0}^m a_i=1$. Note, however, that
only the BDF-$m$ schemes with $m\le 6$ are $A(\alpha)$-stable, while they
are instable for $m>6$.


\subsection{Runge-Kutta schemes}

The Runge-Kutta scheme reads as follows:
\begin{equation}\label{5.rk}
\begin{aligned}
  & \frac{1}{\tau}(n_k-n_{k-1}) = \sum_{i=1}^s b_i K_i, \quad
	K_i = \diver(\na m_i - m_i\na B_\alpha*m_i), \\ 
	& m_i = n_{k-1} + \tau\sum_{j=1}^s a_{ij}K_j, \quad i=1,\ldots,s,
\end{aligned}
\end{equation}
where $s\in\N$ is the number of stages, $b_i\ge 0$ are the weights, and $a_{ij}$
are the Runge-Kutta coefficients. We assume that $\sum_{i=1}^s b_i=1$.
The existence of solutions is only shown for two particular Runge-Kutta schemes;
see below.

First, we claim that the mass is conserved in the following sense.

\begin{lemma}[Conservation of mass]\label{lem.mass}
Let $n_k\in L^1(\R^2)$ be a solution to \eqref{5.rk} such that $m_i\in L^1(\R^2)$ and
$$
  \frac{1}{\tau}\int_{\R^2}(n_k-n_{k-1})\phi dx
	= \sum_{i=1}^s\int_{\R^2}b_i\big(m_i\Delta\phi 
	+ m(\na B_\alpha*m)\cdot\na\phi\big)dx
$$
for all $\phi\in C_0^\infty(\R^2)$. Then
$$
  \int_{\R^2}m_i dx = \int_{\R^2}n_{k}dx = \int_{\R^2}n_{k-1} dx, \quad i=1,\ldots,s.
$$
\end{lemma}

Note, however, that we do not know whether $m_i\ge 0$ in $\R^2$.

\begin{proof}
Using the mollifier $\rho^\eps$ and the cut-of function $\phi_R(x)=\phi(|x|/R)$, where
$\phi$ is defined in \eqref{2.cutoff}, as a test function in the weak formulation 
of the equation for $K_i$ and performing the limit $\eps\to 0$, we find that
$$
  \int_{\R^2} K_i\phi_R dx = \int_{\R^2}m_i\Delta\phi_R dx
	+ \int_{\R^2}m_i(\na B_\alpha*m_i)\cdot\na\phi_R dx.
$$
According to \eqref{2.phi}, the first term on the right-hand side can be
estimated as
$$
  \bigg|\int_{\R^2}m_i\Delta\phi_R dx\bigg| \le \frac{C_2}{R^2}\|m_i\|_{L^1(\R^2)}.
$$
For the second term, we use formulation \eqref{2.naB} of $\na B_\alpha$,
the symmetry argument, and the Lipschitz estimate \eqref{2.phi} for $\na\phi_R$,
which leads to
\begin{align*}
  \bigg|\int_{\R^2} & m_i(\na B_\alpha*m_i)\cdot\na\phi_R dx\bigg| \\
	&= \frac12\bigg|\int_{\R^2}m_i(x)\int_{\R^2}\big(\na\phi_R(x)-\na\phi_R(y)\big)
	\cdot\frac{x-y}{|x-y|}g_\alpha(|x-y|)m_i(y)dy dx\bigg| \\
	&\le \frac{C_1}{2R^2}\int_{\R^2}|m_i(x)|\int_{\R^2}|m_i(y)|dy dx 
	= \frac{C_1}{2R^2}\|m_i\|_{L^1(\R^2)}^2.
\end{align*}
We deduce that for $R\to\infty$, $\int_{\R^2}K_idx=0$. Hence,
\begin{align*}
  \int_{\R^2}m_idx &= \int_{\R^2}n_{k-1}dx 
	+ \tau\sum_{j=1}^s a_{ij}\int_{\R^2}K_j dx = \int_{\R^2}n_{k-1}dx, \\
  \int_{\R^2}n_k dx &= \int_{\R^2}n_{k-1} + \tau\sum_{i=1}^s b_i\int_{\R^2}K_i dx
	= \int_{\R^2}n_{k-1}dx,
\end{align*}
which concludes the proof.
\end{proof}

We are able to show finite-time blow-up for all Runge-Kutta schemes if $\alpha=0$.

\begin{theorem}[Blow-up for Runge-Kutta schemes]\label{thm.bu.rk}
Let $\alpha=0$. Assume that
$$
  n_0\ge 0, \quad I_0 := \int_{\R^2}n_0(x)|x|^2dx < \infty, \quad
	M:=\int_{\R^2}n_0 dx > 8\pi.
$$
Let $(n_k)\subset L^1(\R^2)\cap H^1(\R^2)$ 
be a sequence of nonnegative weak solutions to \eqref{5.rk}. Then this
sequence is finite with maximal index $k_{\rm max}$ defined in \eqref{3.kmax1}.
\end{theorem}

\begin{proof}
Using an approximation of $|x|^2$ as a test function in \eqref{5.rk}
and passing to the de-regularization limit
(see step 5 of the proof of Theorem \ref{thm.bu.ie}), we find that
$$
  I_k-I_{k-1}
	= \tau\sum_{i=1}^s b_i\bigg(4\int_{\R^2}m_i dx - \frac{1}{\pi}\int_{\R^2}m_i(x)
	\int_{\R^2}\frac{x\cdot(x-y)}{|x-y|^2}m_i(y)dydx\bigg).
$$
By Lemma \ref{lem.mass}, the symmetry argument, and $\sum_{i=1}^s b_i=1$,
\begin{align*}
  I_k-I_{k-1} &= \tau\sum_{i=1}^s b_i\bigg(4M-\frac{1}{2\pi}\int_{\R^2}m_i(x)
	\int_{\R^2}m_i(y)dydx\bigg) \\
	&= \tau\sum_{i=1}^s b_i\bigg(4M - \frac{M^2}{2\pi}\bigg) = \frac{M}{2\pi}(8\pi-M).
\end{align*}
Now, we argue as in the proof of Theorem \ref{thm.bu.ie} to conclude.
\end{proof}

The case $\alpha>0$ is more delicate since $m_i\ge 0$ is generally not guaranteed.
Indeed, it follows that (see \eqref{3.aux} and \eqref{2.estg})
\begin{align*}
  I_k-I_{k-1} &= \tau\sum_{i=1}^s b_i\bigg(\frac{M}{2\pi}(8\pi-M)
	+ \frac{1}{2\pi}\int_{\R^2}\int_{\R^2}\big(1-g_\alpha(|x-y|)\big)m_i(x)m_i(y)dydx
	\bigg) \\
	&\le \tau\sum_{i=1}^s b_i\bigg(-\gamma + \frac{\sqrt{\alpha}}{2\pi}M
	\int_{\R^2}|y||m_i(y)|dy\bigg),
\end{align*}
where we recall that $\gamma=M(M-8\pi)/(2\pi)$. By the Cauchy-Schwarz inequality,
$$
  I_k-I_{k-1} \le \tau\sum_{i=1}^s b_i\bigg\{-\gamma + \frac{\sqrt{\alpha}}{2\pi}M
	\bigg(\int_{\R^2}|m_i(y)|dy\bigg)^{1/2}
	\bigg(\int_{\R^2}|y|^2|m_i(y)| dy\bigg)^{1/2}\bigg\},
$$
and this cannot be estimated further as $m_i\ge 0$ may not hold.
However, for the midpoint and trapezoidal rule, we are able to give a result.

\subsection*{Implicit midpoint rule}

The implicit midpoint rule is defined by $s=1$, $a_{11}=1/2$, and $b_1=1$. Then
\eqref{5.rk} becomes
\begin{align*}
  \frac{1}{\tau}(n_k-n_{k-1}) &= \diver(\na m_1-m_1\na B_\alpha*m_1), \\
	m_1 &= n_{k-1} + \frac{\tau}{2}\diver(\na m_1-m_1\na B_\alpha*m_1),
\end{align*}
and since $m_1=\frac12(n_k+n_{k-1})$, this can be rewritten as
\begin{equation}\label{imp}
  \frac{1}{\tau}(n_k-n_{k-1}) = \diver\bigg(\na\bigg(\frac{n_k+n_{k-1}}{2}\bigg)
	- \frac{n_k+n_{k-1}}{2}\na B_\alpha*\frac{n_k+n_{k-1}}{2}\bigg).
\end{equation}

\begin{lemma}[Existence for the midpoint scheme]\label{lem.imp}
Let $\alpha\ge 0$, $n_{k-1}\in W^{1,1}(\R^2)\cap W^{1,\infty}(\R^2)$, and
$$
  \tau\le \frac{2}{(\pi+\frac12)^2}\frac{1}{(\|n_{k-1}\|_X 
	+ \frac{\pi b}{2}\|n_{k-1}\|_X^2 + \frac{\pi}{2}\|\na n_{k-1}\|_X+1)^4}.
$$
Then there exists a unique weak solution 
$n_k\in L^1(\R^2)\cap L^\infty(\R^2)\cap H^1(\R^2)$ to \eqref{imp} with the
following properties:
\begin{itemize}
\item conservation of mass: $\int_{\R^2}n_k dx = \int_{\R^2}n_{k-1}dx$,
\item control of second moment: if $\int_{\R^2}n_{k-1}|x|^2dx<\infty$ then
$\int_{\R^2}n_k|x|^2dx<\infty$.
\end{itemize}
Moreover, if $\alpha>0$ and $n_{k-1}\in Y:=W^{1,1}(\R^2)\cap W^{1,\infty}(\R^2)
\cap H^3(\R^2)$, then $n_k\in Y$.
\end{lemma}

Note that our technique of proof requires higher regularity for $n_{k-1}$
compared to the implicit Euler scheme. For general Runge-Kutta schemes, the
regularity requirement becomes even stronger, which is the reason why we
show existence results only in special cases.

\begin{proof}
We set $n:=n_k$ and $n_0:=n_{k-1}$. For given $\widetilde n\in X$, we solve the
linear problem 
$$
  \bigg(-\Delta+\frac{2}{\tau}\bigg)n = \frac{2}{\tau}n_0
	+ \diver\bigg(\na n_0 - \frac12(\widetilde n+n_0)\na B_\alpha*(\widetilde n+n_0)\bigg)
	\quad\mbox{in }\R^2.
$$
By Lemma \ref{lem.ell}, this problem has a unique solution $n\in H^1(\R^2)$, and
it can be represented by
\begin{equation}\label{3.Tn}
  T[\widetilde n]:= n = \frac{2}{\tau}B_{2/\tau}*n_0 + \na B_{2/\tau}*\na n_0
	- \frac12\na B_{2/\tau}*\big((\widetilde n+n_0)\na c[\widetilde n+n_0]\big),
\end{equation}
writing $\na c[n]=\na B_\alpha*n$ as in section \ref{sec.ie}.
This defines the fixed-point operator $T:S\to S$, 
where $S=\{n\in X:\|n\|_X\le C_B\}$ and
$$
  C_B = \|n_{0}\|_X + \frac{\pi b}{2}\|n_{0}\|_X^2
	+ \frac{\pi}{2}\|\na n_{0}\|_X + 1.
$$
It holds $T(S)\subset S$ since, using similar arguments as in the proof of Theorem
\ref{thm.ie} and the smallness condition on $\tau$,
\begin{align*}
  \|T[n]\|_X &\le \frac{2}{\tau}\|B_{2/\tau}\|_{L^1(\R^2)}\|n_0\|_X
	+ \|\na B_{2/\tau}\|_{L^1(\R^2)}\|\na n_0\|_X \\
	&\phantom{xx}{}
	+ \frac12\|\na B_{2/\tau}\|_{L^1(\R^2)}\|n+n_0\|_X\|\na c[n+n_0]\|_X \\
	&\le \|n_0\|_X + \frac{\pi\sqrt{\tau}}{2\sqrt{2}}\|\na n_0\|_X
	+ \frac{\pi b\sqrt{\tau}}{4\sqrt{2}}\|n+n_0\|_X^2 \\
  &\le C_B.
\end{align*}
We claim that $T:S\to S$ is a contraction. Indeed, let $n$, $m\in S$. Then
\begin{align*}
  \|T[n]-T[m]\|_X &\le \frac12\Big\|\na B_{2/\tau}*(n+n_0)\na B_\alpha*(n+n_0) \\
	&\phantom{xx}{}- \na B_{2/\tau}*(m+n_0)\na B_\alpha*(m+n_0)\Big\|_X \\
	&\le \frac{\pi\sqrt{\tau}}{4\sqrt{2}}\big\|(n+n_0)\na B_\alpha*(n+n_0)
	- (m+n_0)\na B_\alpha*(m+n_0)\big\|_X \\
	&\le \frac{\pi\sqrt{\tau}}{4\sqrt{2}}\Big(\|n\na B_\alpha*n-m\na B_\alpha*m\|_X \\
	&\phantom{xx}{}+ \|n_0\na B_\alpha*(n-m)\|_X
	+ \|(n-m)\na B_\alpha*n_0\|_X\Big) \\
	&\le \frac{\pi b\sqrt{\tau}}{\sqrt{2}}(C_B+1)\|n-m\|_X,
\end{align*}
and we have $\pi b\sqrt{\tau}(C_B+1)/2<1$. The Banach fixed-point theorem now
implies that there exists a unique fixed point $n\in X$.

By the same arguments used in step 2 of the proof of Theorem \ref{thm.ie}, we
infer that $n_k\in H^1(\R^2)$. Steps 4 and 5 show the conservation of mass and
the finiteness of the second moment.

It remains to show that if $n_{k-1}\in Y$ then $n_k$ has the same regularity.
By Lemma \ref{lem.bessel}, $n_k\in H^1(\R^2)$ implies that $\na B_\alpha*n_k\in
H^2(\R^2)$. Therefore,
$$
  \bigg(-\Delta+\frac{2}{\tau}\bigg)n_k = \frac{2}{\tau}n_{k-1} + \Delta n_{k-1}
	- \diver\big((n_k+n_{k-1})\na c[n_k]\big) \in L^2(\R^2).
$$
Elliptic regularity then gives $n_k\in H^2(\R^2)$. We bootstrap this argument to
find that $n_k\in H^3(\R^2)\hookrightarrow W^{1,\infty}(\R^2)$. Taking the
$L^1$ norm of the gradient of $n=n_k$ in \eqref{3.Tn} shows that 
$\|\na n_k\|_{L^1(\R^2)}$ can be estimated in terms of the $H^3$ norms of
$n_k$, $n_{k-1}$, and $c[n_k]$. We conclude that $n_k\in W^{1,1}(\R^2)$,
finishing the proof.
\end{proof}

\begin{lemma}[Blow-up for the midpoint scheme]\label{thm.bu.mip}
Let $\alpha>0$. Assume that
$$
  n_0\ge 0, \quad I_0 := \int_{\R^2}n_0(x)|x|^2dx < \infty, \quad
	M:=\int_{\R^2}n_0 dx > 8\pi.
$$
Let $(n_k)\subset L^1(\R^2)\cap H^1(\R^2)$ 
be a sequence of nonnegative weak solutions to \eqref{imp}. 
Suppose that $I_0\le I^*$ and $\tau\le 2\tau^*$ (see \eqref{Istar}). Then this
sequence is finite with maximal index $k_{\rm max}$ defined in \eqref{3.kmax2}.
\end{lemma}

\begin{proof}
Approximating $|x|^2$ as in step 5 of the proof of Theorem \ref{thm.bu.ie}
and using the nonnegativity of $n_k$ and $n_{k-1}$, we can estimate as
\begin{align*}
  I_k - I_{k-1} &= \int_{\R^2}(n_k-n_{k-1})|x|^2dx \\
	&= 4\tau M -\frac{\tau}{2}\int_{\R^2}(n_k+n_{k-1})\na B_\alpha*(n_k+n_{k-1})dx \\
	&= -\tau\gamma + \frac{\tau}{4\pi}\int_{\R^2}\int_{\R^2}\big(1-g_\alpha(|x-y|)\big)
	(n_k+n_{k-1})(x)(n_k+n_{k-1})(y)dydx \\
	&\le -\tau\gamma + \frac{\tau}{4\pi}\sqrt{\alpha}\int_{\R^2}|x-y|
	(n_k+n_{k-1})(x)(n_k+n_{k-1})(y)dydx \\
	&\le -\tau\gamma + \frac{\tau}{4\pi}\sqrt{\alpha}(2M)M^{1/2}
	\big(I_k^{1/2}+I_{k-1}^{1/2}\big) 
	= -\tau\gamma + \frac{\beta}{2}\big(I_k^{1/2}+I_{k-1}^{1/2}\big).
\end{align*}
Setting again $f(s) = \beta\sqrt{s}-\gamma$, it follows that
$$
  I_k-I_{k-1} \le \frac{\tau}{2}\big(f(I_k)+f(I_{k-1})\big).
$$
Again, since $f(I^*)=0$ and $f$ is increasing, we have $f(I_0)\le 0$.
Let $f(I_{k-1})\le 0$. Then
$$
  I_k-I_{k-1}\le \frac{\tau}{2}f(I_k),
$$
and we can proceed as in the proof of Theorem \ref{thm.bu.ie}. 
\end{proof}

\subsection*{Trapezoidal rule}

The (implicit, two-stage) trapezoidal rule is defined by $s=2$,
$a_{11}=a_{12}=\frac12$, $b_1=b_2=\frac12$, which gives the scheme
\begin{equation}\label{trapez}
  \frac{1}{\tau}(n_k-n_{k-1}) = \frac12\diver\big(\na(n_k+n_{k-1})
	+ n_k\na B_\alpha*n_k + n_{k-1}\na B_\alpha*n_{k-1}\big).
\end{equation}
The existence of weak solutions can be shown exactly as in the proof of
Lemma \ref{lem.imp}, therefore we leave the details to the reader.

\begin{proposition}[Finite-time blow-up for the trapezoidal rule]\label{thm.bu.trp}
Let $\alpha>0$. Assume that 
$$
  n_0\ge 0, \quad I_0 := \int_{\R^2}n_0(x)|x|^2dx < \infty, \quad
	M:=\int_{\R^2}n_0 dx > 8\pi.
$$
Let $(n_k)$ be a sequence of nonnegative weak solutions to \eqref{trapez}. 
Suppose that $I_0\le I^*$ and $\tau\le \tau^*$ (see \eqref{Istar}). Then this
sequence is finite with maximal index $k_{\rm max}$ defined in \eqref{3.kmax2}.
\end{proposition}

\begin{proof}
Arguing as in the previous blow-up proofs, we obtain
\begin{align*}
  I_k-I_{k-1} &= 4\tau M -\frac{\tau}{2}\int_{\R^2}n_k\na B_\alpha*n_k dx
	- \frac{\tau}{2}\int_{\R^2}n_{k-1}\na B_\alpha*n_{k-1} dx \\
	&\le -\tau\gamma + \frac{\tau}{4\pi}\sqrt{\alpha}\int_{\R^2}|x-y|n_k(y)n_k(x)dydx \\
	&\phantom{xx}{}
	+ \frac{\tau}{4\pi}\sqrt{\alpha}\int_{\R^2}\int_{\R^2}|x-y|n_{k-1}(y)n_{k-1}(x)dydx
	\\
	&\le -\tau\gamma + \frac{\tau}{4\pi}\sqrt{\alpha}M^{3/2}
	\big(I_k^{1/2}+I_{k-1}^{1/2}\big).
\end{align*}
Now, we proceed as in the proof of Proposition \ref{thm.bu.mip}.
\end{proof}


\section{Numerical examples}\label{sec.num}

The numerical experiments are performed by using the finite-element method
introduced by Saito in \cite{Sai07} and analyzed in \cite{Sai12}.
In contrast to \cite{Sai07}, we choose higher-order temporal approximations.
The scheme uses a first-order upwind technique for the drift term, the
lumped mass method, and a decoupling procedure. We take $\alpha=1$ in \eqref{1.eq}
and consider bounded domains only. Equations \eqref{1.eq} are supplemented with
no-flux boundary conditions. In the first example, the domain is large enough
to avoid effects arising from boundary conditions. The second example, on the
other hand, illustrates blow-up at the boundary.

\subsection{Numerical scheme}

Let ${\mathcal T}_h$ be a triangulation of the bounded set $\Omega\subset\R^2$,
where $h=\max\{\operatorname{diam}(K):K\in{\mathcal T}_h\}$, and let $D_i$ be
the barycentric domain associated with the vertex $P_i$; see
\cite[Section 2]{Sai07} for the definition. Let $\chi_i$ be the
characteristic function on $D_i$ and let $Y_h$ be the span of all $\chi_i$.
Furthermore, let $X_h$ be the space of linear finite elements.
The lumping operator $M_h:X_h\to Y_h$ is defined by
$M_h v_h=\sum_i v_h(P_i)\chi_i$, where $v_h\in X_h$, and the 
mass-lumped inner product is given by 
$$
  (v_v,w_h)_h = (M_hv_h,M_hw_h)_2, \quad v_h,w_h\in X_h,
$$
where $(\cdot,\cdot)_2$ is the $L^2$ inner product. 

For the discretization of the drift term, we define the discrete Green operator
$G_h:X_h\to X_h$ as the unique solution $v_h=G_hf_h\in X_h$ to
$$
  (\na v_h,\na w_h)_2 + (v_h,w_h)_2 = (f_h,w_h)_2\quad\mbox{for }w_h\in X_h.
$$
(Recall that $\alpha=1$.)
The drift term is approximated by the trilinear form $b_h:X_h^3\to\R$,
$$
  b_h(u_h,v_h,w_h) = \sum_i w_h(P_i)\sum_{P_j\in\Lambda_i}
	\big(v_h(P_i)\beta_{ij}^+(u_h) - v_h(P_j)\beta_{ij}^-(u_h)\big),
$$
where $\Lambda_i$ is the set of vertices $P_j$ that share an edge with $P_i$.
For the definition of $\beta_{ij}^\pm$, we first introduce the set
$S_h^{ij}$ of all elements $K\in{\mathcal T}_h$ such that $P_i$, $P_j\in K$
and the exterior normal vector $\nu_{ij}$ to $\pa D_i\cap\pa D_j$ with
respect to $D_i$. Then
$$
  \beta_{ij}^\pm(u_h) = \sum_{K\in S_h^{ij}}
	\mbox{meas}\big((\pa D_i\cap\pa D_j)|_K\big)
	\big[(\na G_hu_h)|_K\cdot\nu_{ij}|_K\big]_\pm,
$$
where $[s]_\pm=\max\{0,\pm s\}$ for $s\in\R$.
It is explained in \cite{Sai07} that the trilinear form approximates the
integral $\int_\Omega v\na(Gu)\cdot\na w dx$, where $Gu$ is the Green
operator associated with $-\Delta+1$ on $L^2$. 

Equation \eqref{1.eqc} is solved in a semi-implicit way. This means,
for the BDF-2 scheme and for given $u_h$, that we solve the linear problem
\begin{equation}\label{4.eq}
  \frac{1}{\tau}\bigg(\frac32 n_h^k - 2n_h^{k-1} + \frac12 n_h^{k-2},w_h\bigg)_h
	+ (\na n_h^k,\na w_h)_2 + b_h(u_h,n_h,w_h) = 0
\end{equation}
for all $w_h\in X_h$. Here, $n_h^k$ is an approximation of $n(\cdot,\tau k)$.
This defines the solution operator $N(u_h)=n_h$ and
the scheme is completed by chosing $u_h$. Saito has taken $u_h=n_h^{k-1}$, 
giving the usual semi-implicit scheme of first order. 
For higher-order schemes, we need to iterate.
For this, we introduce the iteration $u_h^{(0)}:=n_h^{k-1}$ and
$u_h^{(m)}:=N(u_h^{(m-1)})$ for $m\ge 1$. The iteration stops when
$\|u_h^{(m)}-u_h^{(m-1)}\|_{L^\infty(\R^2)}<\eps$ for some tolerance $\eps>0$
or if $m\ge m_{\rm max}$ for a maximal number $m_{\rm max}$ of iterations.

In a similar way, we define the scheme with the midpoint discretization:
$$
  \frac{1}{\tau}(n_{h}^k - n_h^{k-1},w_h)_h 
	+ \frac{1}{2} (\na(n_{h}^k+n_{h}^{k-1}),\na w_h)_2
	+ \frac{1}{4} b_h(u_h+n_{k-1},n_{h}^k+n_{h}^{k-1},w_h) = 0 
$$
for all $w_h \in X_h$.
The resulting linear systems are computed by using MATLAB.
We choose the domain $\Omega=(0,1)^2$, which is triangulated uniformly by $2a^2$
triangles with maximal size $h=\sqrt{2}/a$.
The numerical parameters are $a=64$, $\tau=5\cdot 10^{-5}$,
$\eps=10^{-4}$, and $m_{\rm max}=10$ if not stated otherwise.

\subsection{Numerical results}

To illustrate the behavior of the solutions, we choose the initial data
as a linear combination of the shifted Gaussians
$$
  W_{x_0,y_0}(x,y) = \frac{M}{2\pi\theta}
	\exp\bigg(-\frac{(x-x_0)^2+(y-y_0)^2}{2\theta}\bigg),
$$
where $(x_0,y_0)\in(0,1)^2$, $M>0$, and $\theta>0$. 
Clearly, the mass of $W_{x_0,y_0}$ equals $M$.
For the first example, we choose $\theta=1/500$, $M=6\pi$, and
$$
  n_0 = W_{0.33,0.33} + W_{0.33,0.66}	+ W_{0.66,0.33} + W_{0.66,0.66}.
$$
The initial mass is $24\pi>8\pi$ and thus, we expect the solutions to blow up
in finite time. 

Figure \ref{fig.test1} shows the cell density $n(x,t)$ at various time
instances computed from the BDF-2 scheme. 
As expected, the solution blows up in finite time in the center of
the domain. Note that the numerical
solution is always nonnegative and conserve the total mass.

\begin{figure}[ht]
\includegraphics[width=75mm]{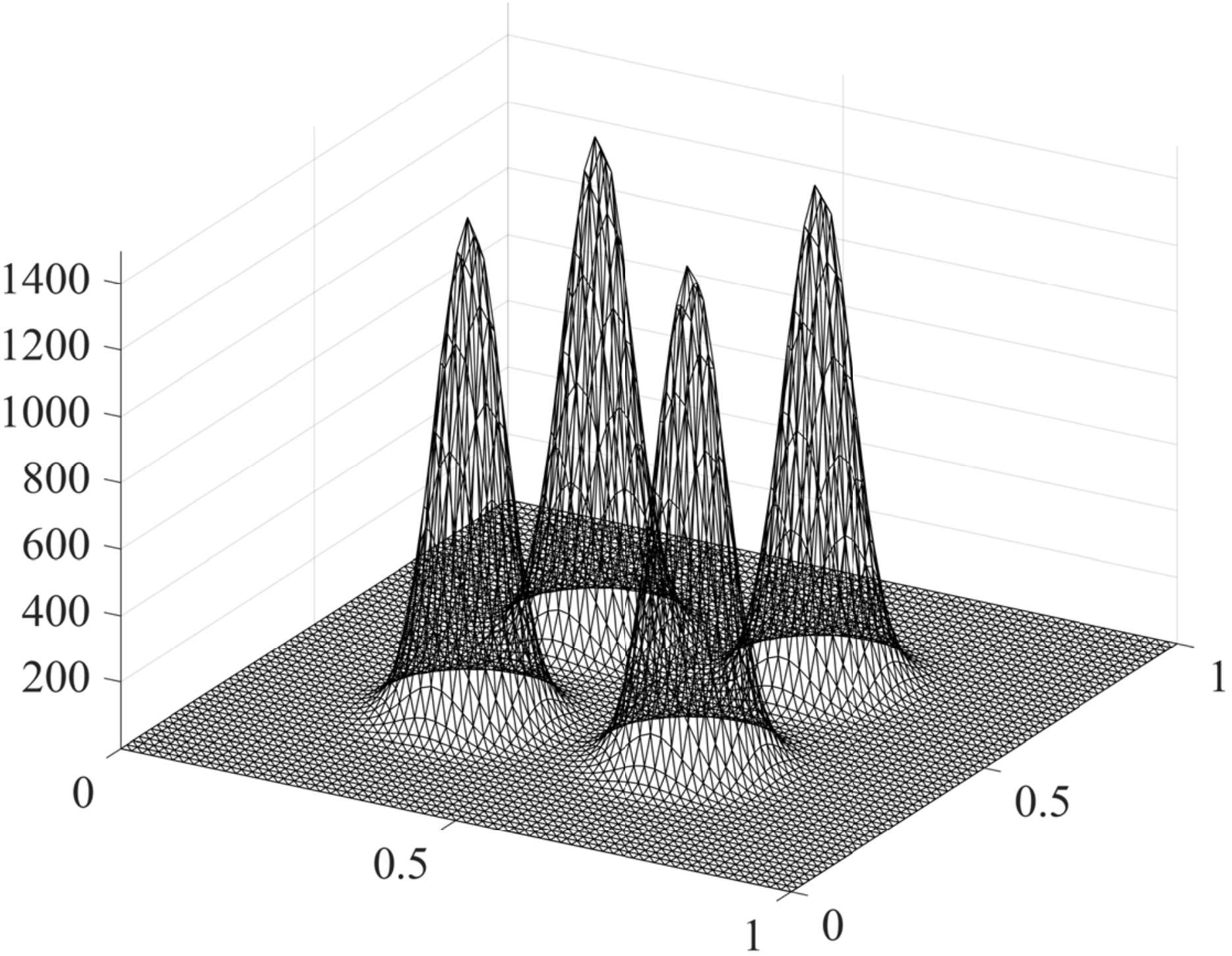}
\includegraphics[width=75mm]{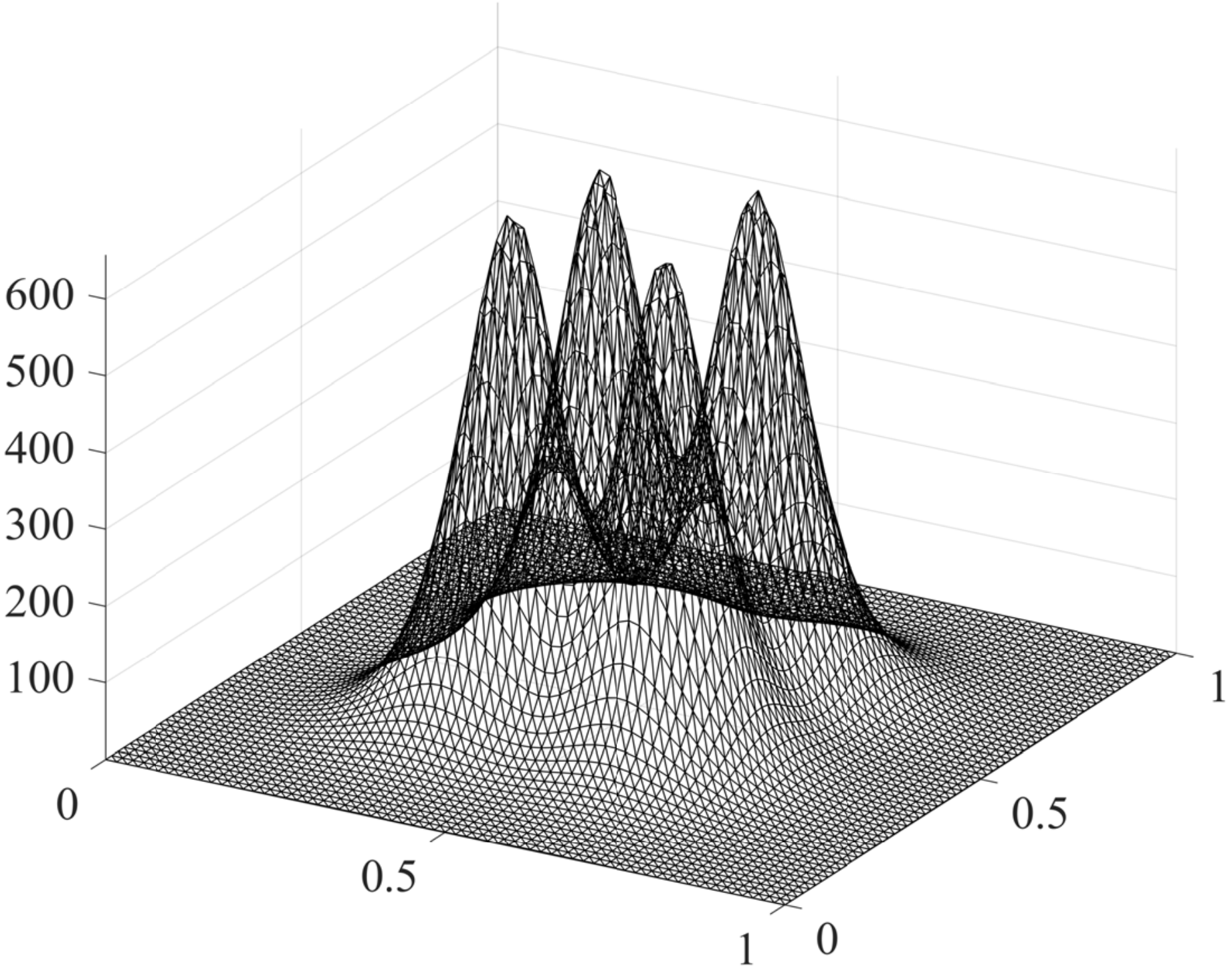}
\includegraphics[width=75mm]{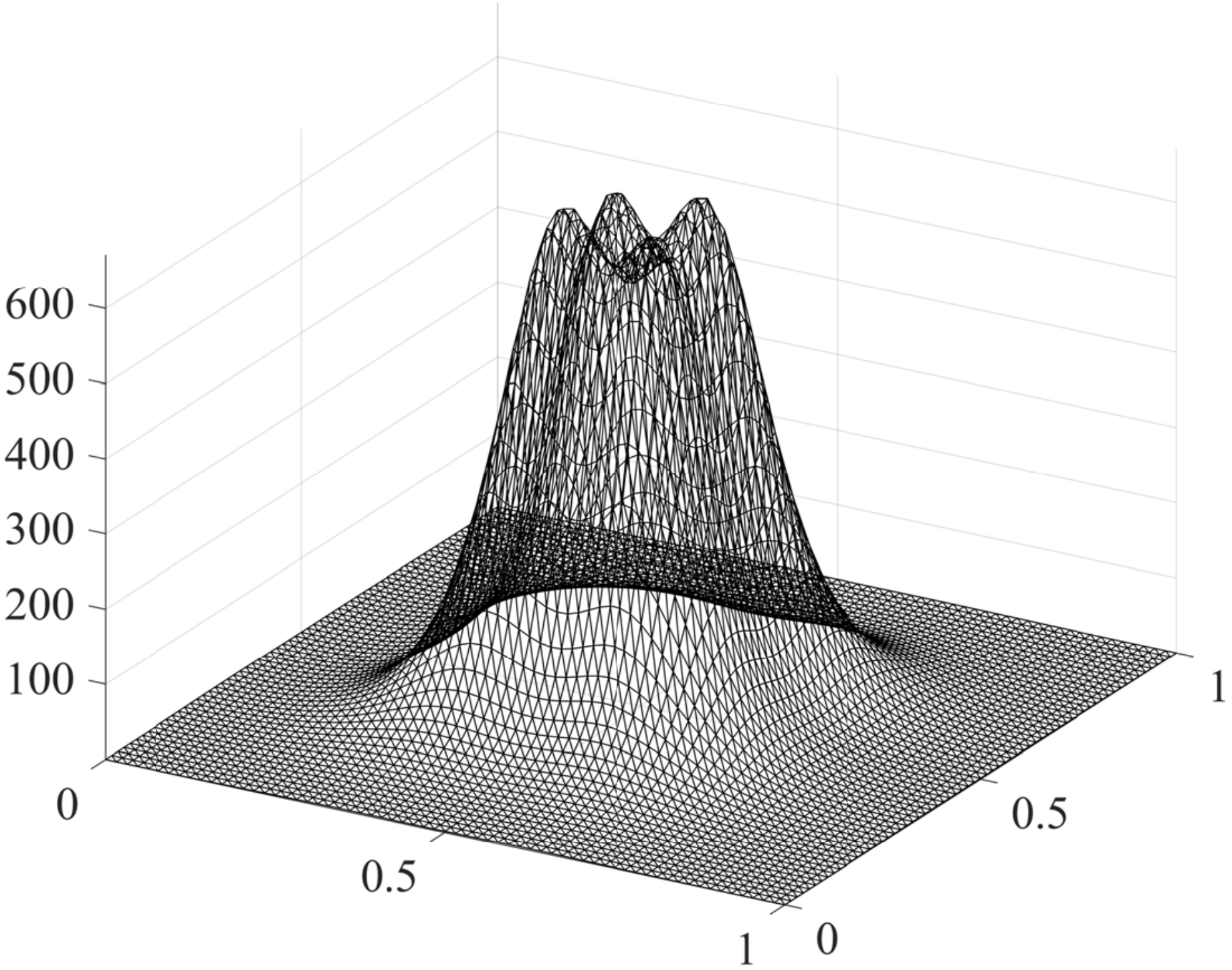}
\includegraphics[width=75mm]{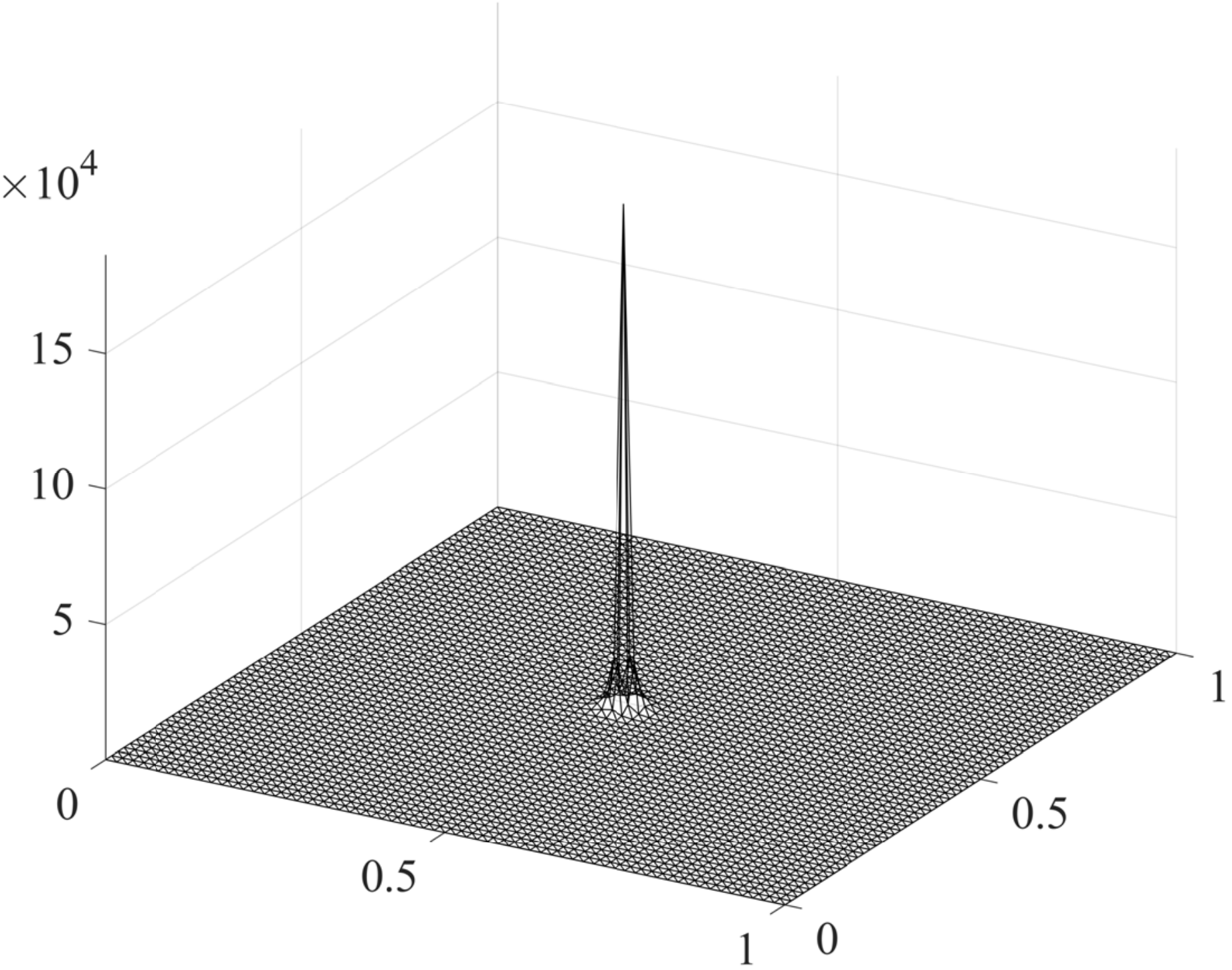}
\caption{Cell density computed from the BDF-2 scheme at times
$t=0$ (top left), $t=0.005$ (top right), $t=0.007$ (bottom left), 
$t=0.02$ (bottom right).}
\label{fig.test1}
\end{figure}

A nonsymmetric situation is given by the initial data
$$
  n_0 = \frac13W_{0.33,0.66} + \frac12 W_{0.33,0.33} + W_{0.66,0.66},
$$
taking $\theta=1/500$, $M=6\pi$, and the same numerical parameters as above. 
The total mass of $n_0$ is $11\pi$, so we expect
again blow up of the solutions. This is illustrated in Figure \ref{fig.test2}.
The solution aggregates, moves to the boundary and blows up.
Again, the discrete solution stays nonnegative and conserves the mass.

\begin{figure}[ht]
\includegraphics[width=75mm]{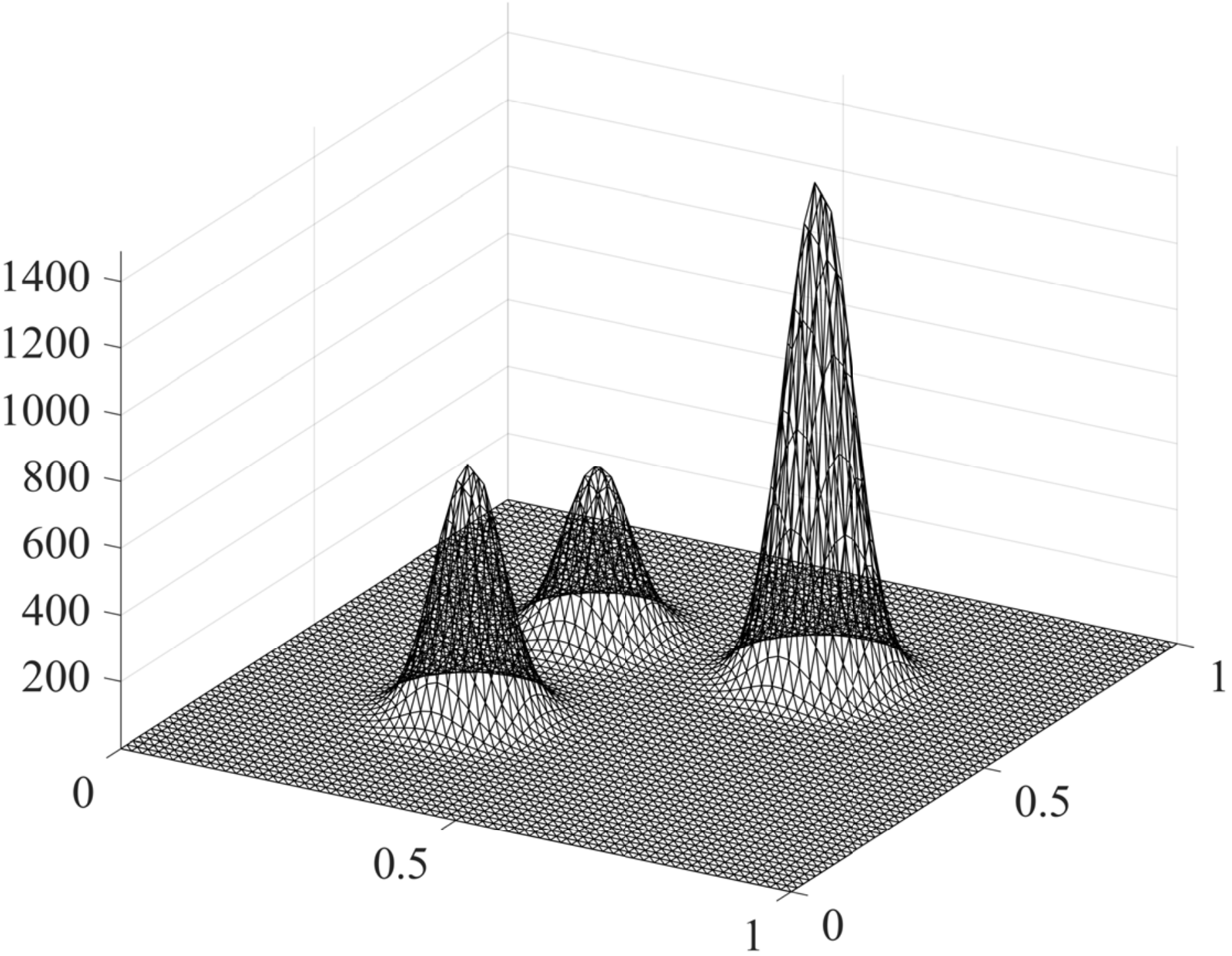}
\includegraphics[width=75mm]{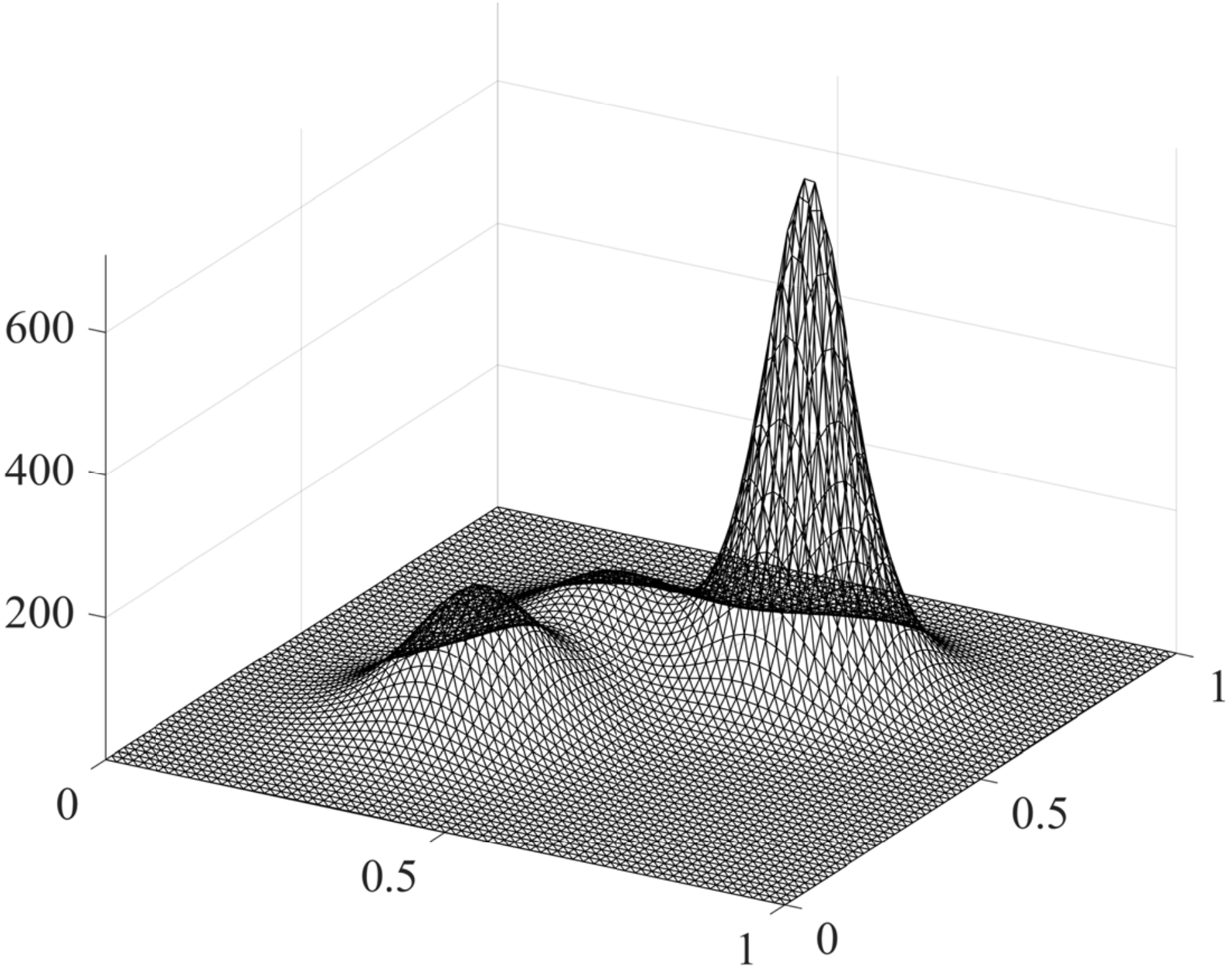}
\includegraphics[width=75mm]{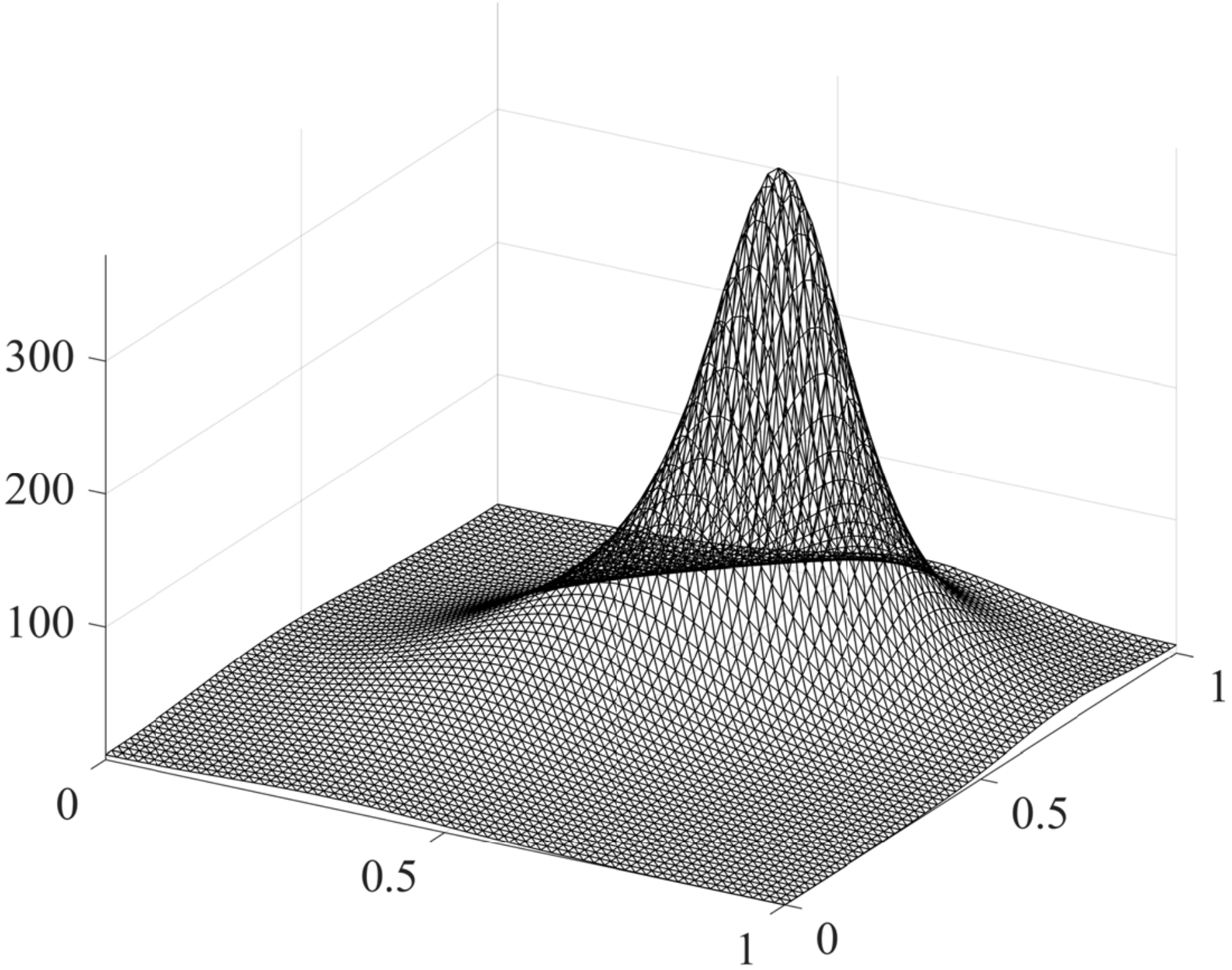}
\includegraphics[width=75mm]{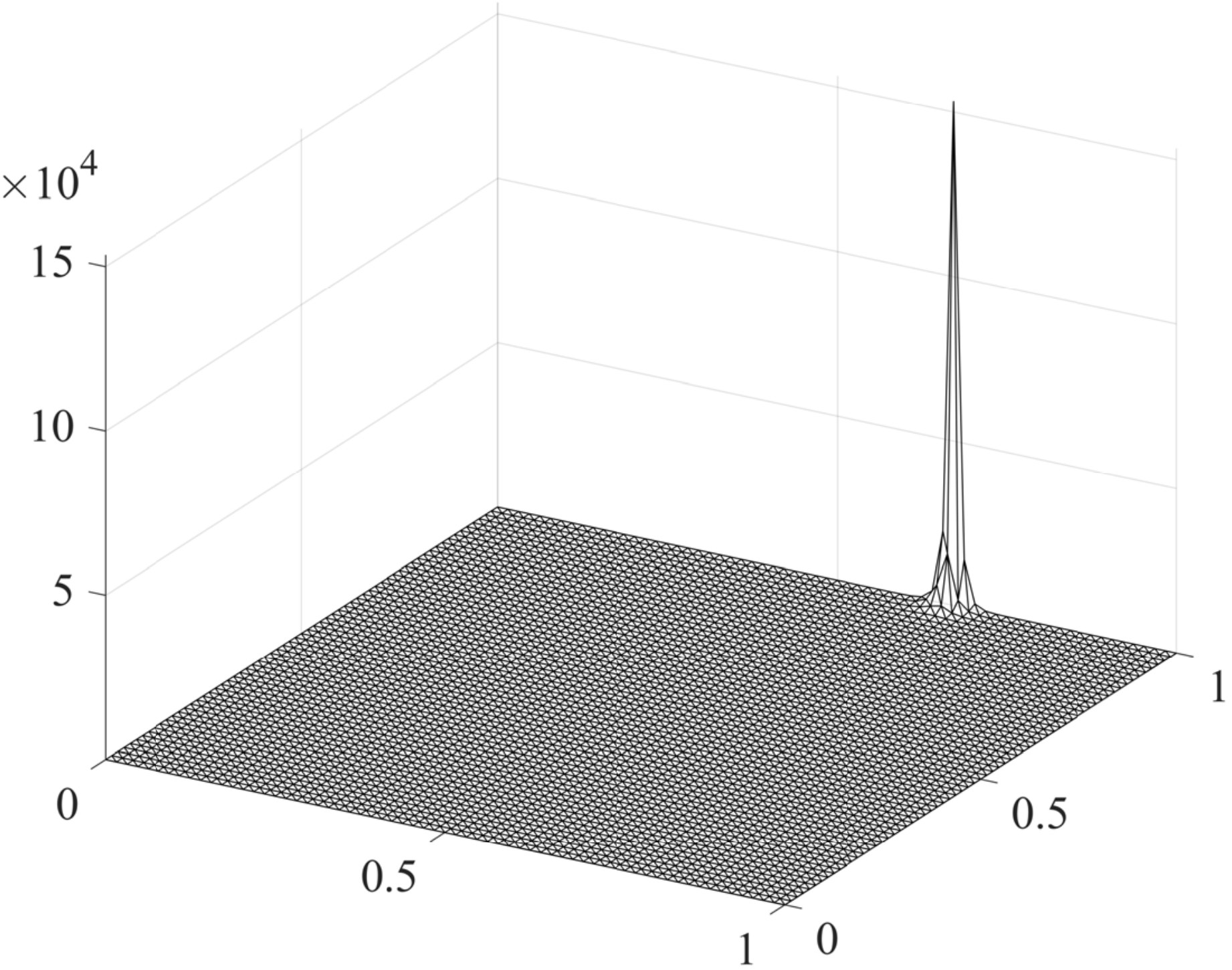}
\caption{Cell density computed from the BDF-2 scheme at times
$t=0$ (top left), $t=0.005$ (top right), $t=0.02$ (bottom left), 
$t=0.1001$ (bottom right).}
\label{fig.test2}
\end{figure}

\subsection{Convergence rate}

To calculate the temporal convergence rates and to show that the schemes 
are indeed of second order, we compute a reference solution $n_{\rm ref}$
with the very small time step
$\tau=10^{-6}$ and compare it in various $L^p$ norms with the solutions $n_\tau$
using larger time step sizes $\tau$. 
We choose the same initial datum as in the first example with
$M=24\pi$. Figure \ref{fig.conv} shows the $L^p$ error
$$
  e_p = \|n_\tau(\cdot,T)-n_{\rm ref}(\cdot,T)\|_{L^p(\Omega)},
$$
where the end time $T=0.01$ is chosen such that the density already started
to aggregate but blow up still did not happen. As expected, the $L^p$ errors
are approximately of second order.

\begin{figure}[ht]
\includegraphics[width=75mm]{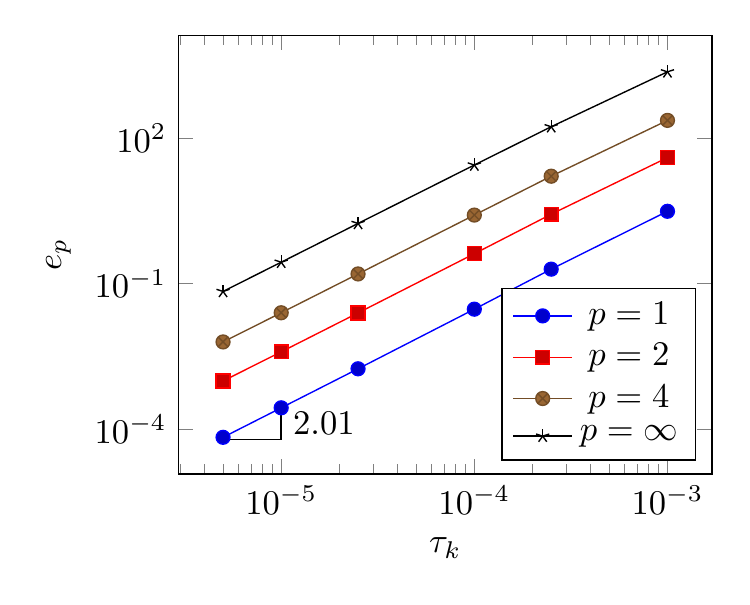}
\includegraphics[width=75mm]{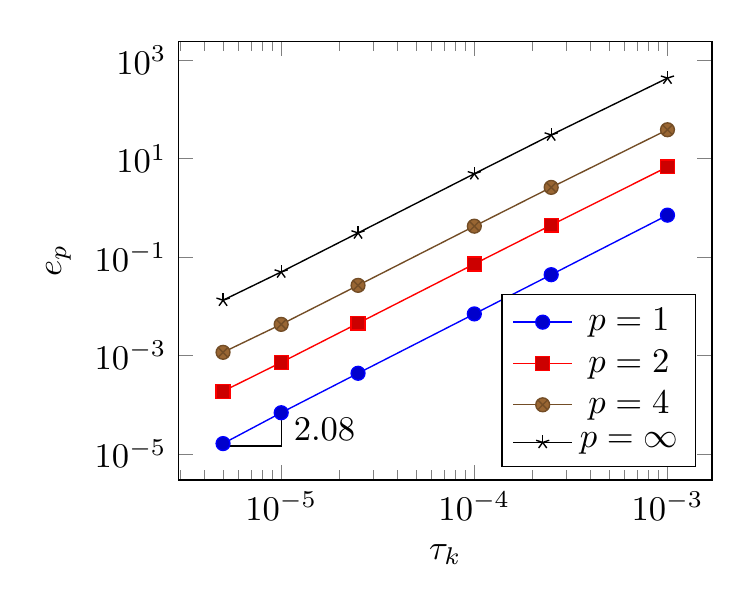}
\caption{$L^p$ error $e_p$ for $p=1,2,4,\infty$ at time $T=0.01$ 
for various time step sizes $\tau_k=\tau$ (left: BDF-2 discretization; right:
midpoint discretization).}
\label{fig.conv}
\end{figure}

\subsection{Numerical blow-up}

We demonstrate that the bound for the blow-up time $T^*=\tau k_{\rm max}$
derived in the time-discrete situation can serve as a bound for the numerical
blow up. It is well known that the computation of the numerical blow-up time
is rather delicate. For instance, Chertock et al.\ \cite{CEHK16} use the
$L^\infty$ norm of the density as a measure of the numerical blow-up time,
since $\|n\|_{L^\infty(\Omega)}$ is proportional to $h^2$ (recall that
$h$ is the spatial grid size). Numerical blow-up may be reached
when the numerical solution becomes negative \cite{ChKu08,SSKT10} or when
the seond moment becomes negative \cite{HaSc09}.
However, since our scheme conserves the mass and the grid is finite,
the numerical solution cannot blow up in the $L^\infty$ norm. Instead,
the solution converges to a state where the mass concentrates at certain points
and no further growth is possible.
Moreover, as the scheme conserves the nonnegativity numerically, 
the second moment cannot become negative, but it will be small near blow-up. 
A lower bound for the blow-up time was derived in, for instance, 
\cite[Theorem~2.2]{FMV15} in two space dimensions (but with a nonexplicit bound) and in
\cite[Prop.~3.1]{CCE12} in three space dimensions.

For the numerical test, we choose the initial datum $n_0=W_{1,1}$
on the domain $\Omega=(0,2)^2$ with parameters $\theta=1/500$, $M=30\pi$, 
and $\tau=10^{-5}$. The grid sizes are $h=0.02$, $0.04$, $0.08$.
Figure \ref{fig.blowup} illustrates the evolution of $\|n_k\|_{L^\infty(\Omega)}$
and $I_k$. The vertical line marks the bound $k_{\rm max}$ from \eqref{3.kmax2}.
We observe that the $L^\infty$ norm and the second moment reach a limit 
close to $k_{\rm max}$. The initial density and the density at $k_{\rm max}$
are displayed in Figure \ref{fig.dens} for comparison.

\begin{figure}[ht]
\includegraphics[width=75mm]{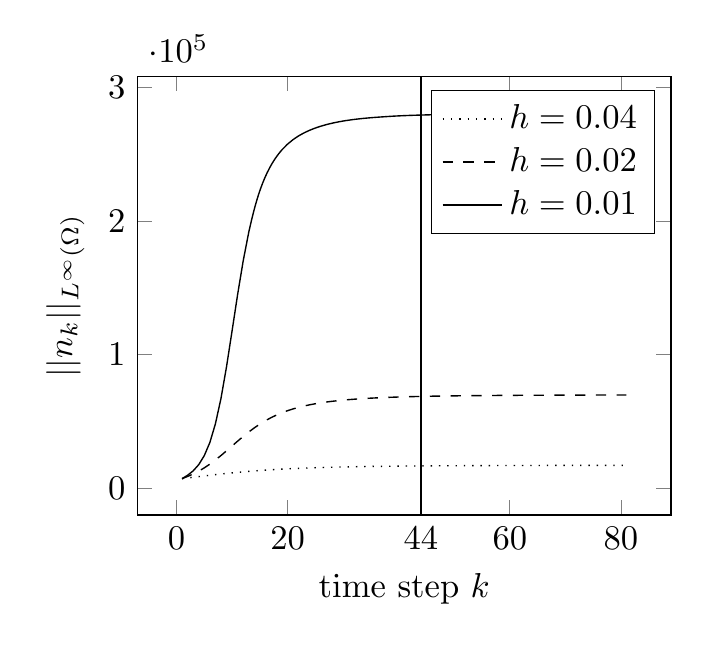}
\includegraphics[width=75mm]{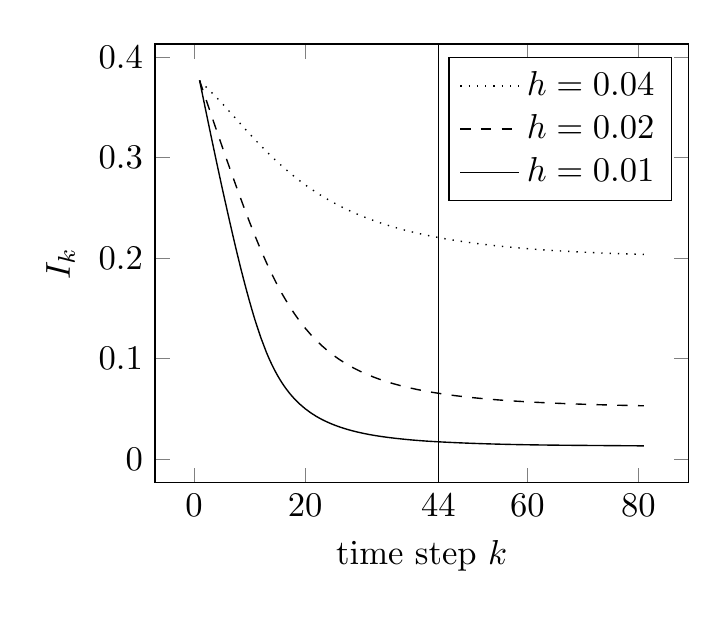}
\caption{$L^\infty$ norm $\|n_k\|_{L^\infty(\Omega)}$ (left) and second moment
$I_k$ (right) versus time. The vertical line marks the
upper bound $k_{\rm max}$ defined in \eqref{3.kmax2}.}
\label{fig.blowup}
\end{figure}

\begin{figure}[ht]
\includegraphics[width=75mm]{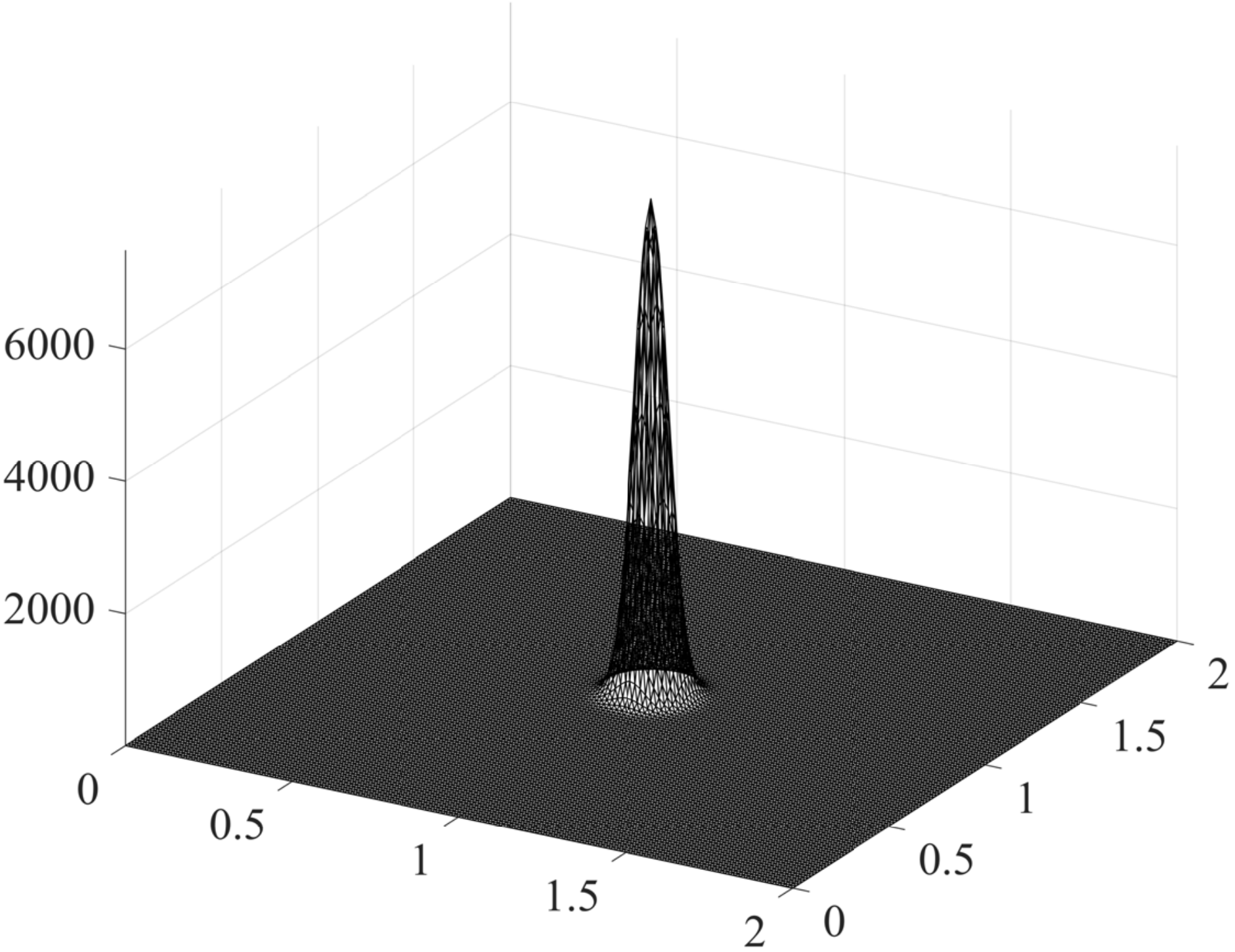}
\includegraphics[width=75mm]{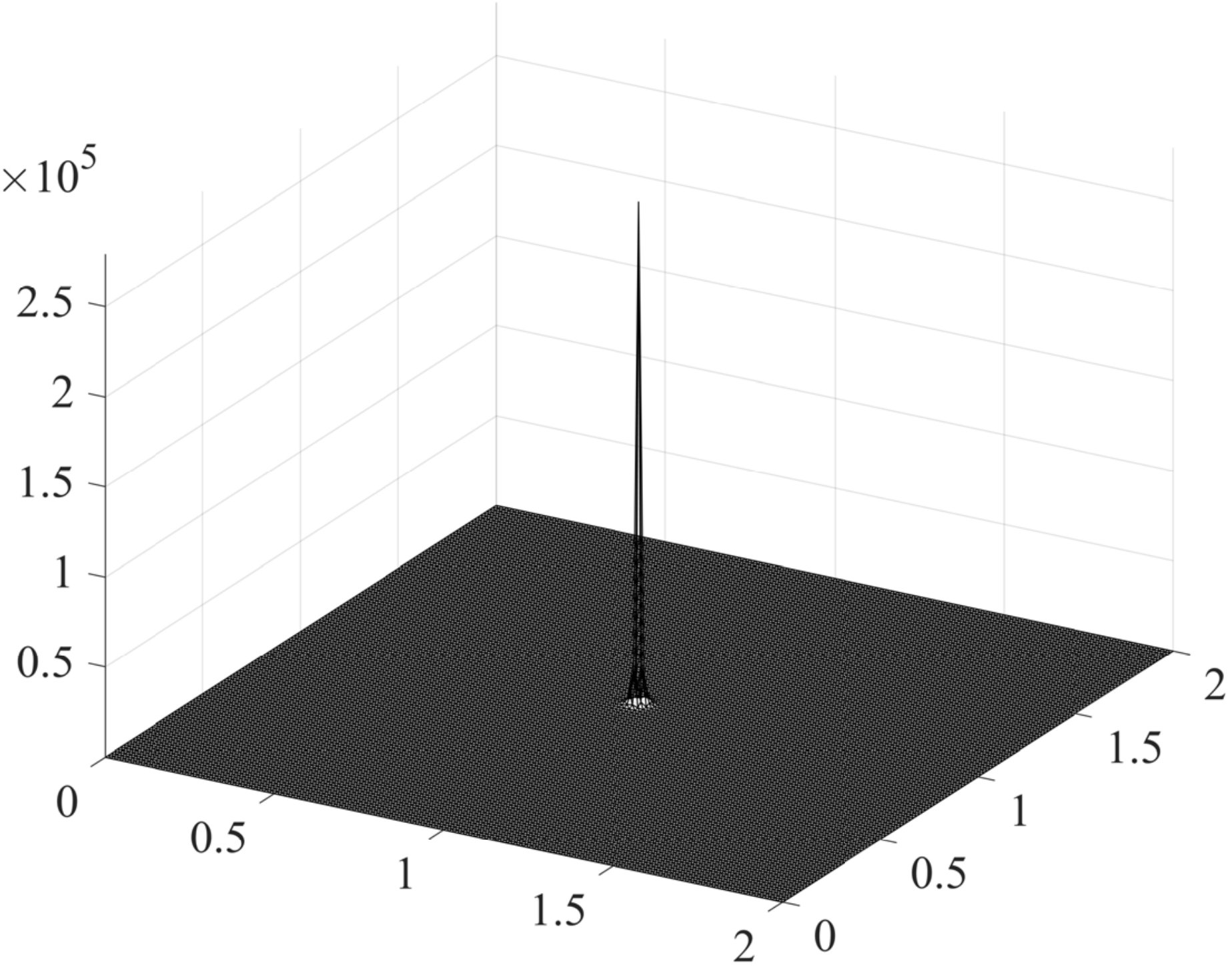}
\caption{Cell density at time step $k=0$ (left) and $k=k_{\rm max}=44$.
The mesh size is $h=0.02$.}
\label{fig.dens}
\end{figure}


\begin{appendix}
\section{Some auxiliary results}

We recall that the function
$$
  B_\alpha(x) = \begin{cases} 
  -\frac{1}{2 \pi} \log|x| & \mbox{for }\alpha=0, \\
  \frac{1}{(4 \pi)^{n/2}} \int_0^{\infty} 
	t^{-n/2}e^{-\alpha t-|x|^2/(4t)}dt & \mbox{for }\alpha >0,
   \end{cases}
$$
defined for $x\in\R^n$, is called the Newton potential if $\alpha=0$ and the 
Bessel potential if $\alpha\neq 0$. 
We need the following properties of the Bessel potential.

\begin{lemma}[Bessel potential]\label{lem.bessel}
Let $\alpha>0$ and $k\in\N_0$. 
Then $B_\alpha$ is a fundamental solution of the operator
$-\Delta+\alpha$. For given $f\in H^k(\R^n)$, the function 
$u=B_\alpha*f\in H^{k+2}(\R^n)$ solves
$$
  -\Delta u + \alpha u = f \quad\mbox{in }\R^n.
$$
Furthermore, it holds that $D^\beta(B_\alpha*f)=B_\alpha*D^\beta f$ for all
multi-indices $\beta\in\N^n$, $|\beta|\le k$ and
\begin{align}
  & B_\alpha\ge 0, \quad \|B_\alpha\|_{L^1(\R^n)} = \frac{1}{\alpha}, \label{B1} \\
  & \|\na B_\alpha\|_{L^1(\R^n)} = \frac{C(n)}{\pi^{(n-1)/2}\alpha^{1/2}}, \quad
	\|\na B_\alpha\|_{L^1(\R^2)} = \frac{\pi}{2\alpha^{1/2}}, \label{B2}
\end{align}
where the constant $C(n)>0$ only depends on $n$.
\end{lemma}

\begin{proof}
We only prove \eqref{B2}, since the other properties are standard; see, e.g.,
Theorem 1.7.1, Corollary 1.7.2, and Examples 12.5.8 in \cite{Kry08}. By
Fubini's theorem and the substitution $u = x/\sqrt{4t}$, we find that
\begin{align*}
  \|\na B_\alpha\|_{L^1(\R^n)} 
	&= \frac{1}{2(4t)^{n/2}}\int_0^\infty
	t^{n/2+1}e^{-\alpha t}\int_{\R^n} e^{-|x|^2/(4t)}|x|dxdt \\
	&= \frac{1}{\pi^{n/2}}\int_0^\infty t^{-1/2}e^{-\alpha t}dt
	\int_{R^n}e^{-|u|^2}|u|du = \frac{1}{\pi^{(n-1)/2}}\alpha^{-1/2}C(n),
\end{align*}
where $C(n)=\int_{\R^n}e^{-|u|^2}|u|du$. In particular, when $n=2$, we obtain
$$
  C(2) = \int_0^\infty\int_0^{2\pi} e^{-r^2} r^2 d\phi dr = \frac{\pi^{3/2}}{2},
$$
ending the proof.
\end{proof}

\begin{lemma}[Young's inequality]\label{lem.young}
Let $g\in L^q(\R^n)$, $h\in L^r(\R^n)$ for $1\le q,r\le\infty$, and
$1/q+1/r=1/p+1$. Then $g*h\in L^p(\R^n)$ and
$$
  \|g*h\|_{L^p(\R^n)} \le \|g\|_{L^q(\R^n)}\|h\|_{L^r(\R^n)}.
$$
\end{lemma}

\begin{lemma}[Elliptic problem]\label{lem.ell} 
Let $\tau>0$, $f\in L^2(\R^2)^2$, and $g\in L^2(\R^2)$. Then there exists a unique
weak solution $n\in H^1(\R^2)$ to
\begin{equation}\label{a.n1}
  -\Delta n + \tau^{-1}(n-g) = -\diver f\quad 
	\mbox{in }\R^2,
\end{equation}
and this solution can be represented as
\begin{equation}\label{a.n2}
  n = \frac{1}{\tau}B_{1/\tau}*g - \na B_{1/\tau}*f \quad\mbox{in }\R^2.
\end{equation}
\end{lemma}

Equation \eqref{a.n1} correspond to the implicit Euler discretization of a
parabolic problem with $n$ being the solution at the actual time step
and $g$ being the solution at the previous time step. Although the result is
standard, we give proof for the sake of completeness.

\begin{proof}
Let $f_k\in C_0^\infty(\R^2)^2$ be 
such that $f_k\to f$ in $L^2(\R^2)^2$ as $k\to\infty$.
By Lemma \ref{lem.bessel}, there exists a unique solution $n_k\in H^1(\R^2)$ to
\begin{equation}\label{a.aux}
  -\Delta n_k + \tau^{-1}n_k = \tau^{-1} g - \diver f_k,
\end{equation}
and, by the variation-of-constants formula and integration by parts,
\begin{equation}\label{a.aux2}
  n_k(x) = \frac{1}{\tau}(B_{1/\tau}*g)(x) - \int_{\R^2}(\na B_{1/\tau})(x-y)\cdot 
	f_k(y)dy.
\end{equation}
Taking the test function $n_k-n_\ell$ in the difference of the weak formulations for
$n_k$, $n_\ell$ corresponding to $f_k$, $f_\ell$, respectively, it follows that
\begin{align*}
  \|\na(n_k-n_\ell)\|^2_{L^2(\R^2)} + \frac{1}{\tau}\|n_k-n_\ell\|_{L^2(\R^2)}^2
	&= \int_{\R^2}(f_k-f_\ell)\cdot\na(n_k-n_\ell)dx \\
  &\le \frac12\|f_k-f_\ell\|_{L^2(\R^2)}^2 + \frac12\|\na(n_k-n_\ell)\|_{L^2(\R^2)}^2.
\end{align*}
Since $(f_k)$ is a Cauchy sequence, $(n_k)$ is a Cauchy sequence in $H^1(\R^2)$ 
and hence there exists a $\widetilde n \in H^1(\R^2)$, such that $n_k\to \widetilde n$ strongly in $H^1(\R^2)$ as $k\to\infty$. Therefore,
we can perform the limit $k\to \infty$ in the weak formulation of \eqref{a.aux} 
leading to \eqref{a.n1}.

It remains to show \eqref{a.n2}. Let
$ n = (1/\tau)B_{1/\tau}*g - \na B_{1/\tau}*f$.
Then, by Lemma \ref{lem.young} and \eqref{a.aux2},
\begin{align*}
  \|\widetilde n-n\|_{L^2(\R^2)}
	&\le \|\widetilde n-n_k\|_{L^2(\R^2)} + \|n_k-n\|_{L^2(\R^2)} \\
  &\le \|\widetilde n-n_k\|_{L^2(\R^2)} + \|\na B_{1/\tau}\|_{L^2(\R^2)}
	\|f_k-f\|_{L^2(\R^2)}.
\end{align*}
The right-hand side can be made arbitrarily small by choosing $k$ sufficiently
large. This shows that $\widetilde n=n$ in $\R^2$.
\end{proof}

\end{appendix}


\end{document}